\documentclass{amsart}
\pdfoutput=1
\usepackage[utf8]{inputenc} 

\usepackage[T1]{fontenc}    
\usepackage{ae,aecompl}     
\usepackage{url}
\usepackage[usenames,dvipsnames]{xcolor}

\usepackage{amsmath}
\usepackage{amsfonts}
\usepackage{amssymb}
\usepackage{amsthm} %
\usepackage{mathrsfs} %
\usepackage{enumerate} 
\usepackage{verbatim}	

\usepackage{tikz}
\usepackage{pgfplots}
\usetikzlibrary{calc}
\usetikzlibrary{shapes}

\usetikzlibrary{intersections}
\usetikzlibrary{patterns}

\usepackage{caption}
\usepackage[hyperfootnotes]{hyperref}
\usepackage{mdframed} 

\theoremstyle{plain} 
\newtheorem{conject}{Conjecture}
\newtheorem{thm}{Theorem}[section]
\newtheorem{cor}[thm]{Corollary}
\newtheorem{prop}[thm]{Proposition}
\newtheorem{lem}[thm]{Lemma}

\theoremstyle{definition}
\newtheorem{defi}[thm]{Definition}

\newtheorem{remark}[thm]{Remark}
\newtheorem{qu}{Question}
\newtheorem{ex}[thm]{Example}

\newcommand{\h}{\widehat}
\newcommand{\cone}{\operatorname{cone}}
\newcommand{\conv}{\operatorname{conv}}
\newcommand{\acc}{\operatorname{Acc}}

\newcommand{\Phip}{\Phi^{{\scriptscriptstyle +}}}
\newcommand{\Phipprime}{\Phi'^{{\scriptscriptstyle +}}}
\newcommand{\Phipim}{\Phi_{im}^{{\scriptscriptstyle +}}}
\newcommand{\PhipI}{\Phi_I^{{\scriptscriptstyle +}}}

\definecolor{Ccolor}{rgb}{0,0.5,0}
\definecolor{JPcolor}{rgb}{0,0,1}

\author[C. Hohlweg]{Christophe~Hohlweg$^{\diamond}$}
\address[Christophe Hohlweg]{Universit\'e du Qu\'ebec \`a Montr\'eal\\
LaCIM et D\'epartement de Math\'ematiques\\ CP 8888 Succ. Centre-Ville\\
Montr\'eal, Qu\'ebec, H3C 3P8\\ Canada}
\email{hohlweg.christophe@uqam.ca}
\urladdr{http://hohlweg.math.uqam.ca}
\thanks{$^\diamond$supported by NSERC Discovery grant {\em Coxeter groups and related structures}.}

\author[J.-P.~Labb\'e]{Jean-Philippe Labb\'e$^{ \ddagger}$} 
\address[J.-P. Labb\'e]{Einstein Institute of Mathematics, Hebrew University of Jerusalem, Jerusalem 91904, Israel}
\email{labbe@math.huji.ac.il}
\urladdr{http://www.math.huji.ac.il/~labbe}
\thanks{$^\ddagger$supported by the FQRNT post-doctoral fellowship and a post-doctoral ISF grant (805/11).}


\title[On inversion sets and the weak order in Coxeter groups]{On inversion sets and the weak order in Coxeter groups}

\begin{document}

\begin{abstract}
In this article, we investigate the existence of joins in the weak order of an infinite Coxeter group $W$. We give a geometric characterization of the existence of a join for a subset $X$ in $W$  in terms of the inversion sets of its elements and their position relative to the imaginary cone. Finally, we discuss inversion sets of infinite reduced words and the notions of biconvex and biclosed sets of positive roots.   
\end{abstract}

\date{\today}


\maketitle

\newcommand\under{\mathbin{\backslash}}

\section{Introduction}

The Cayley graph of a Coxeter system $(W,S)$ is naturally oriented: we orient an edge $w\to ws$ if $w\in W$ and $s\in S$ such that $\ell(ws)>\ell(w)$. Here $\ell(u)$ denotes the {\em length} of a reduced word for $u\in W$ over the alphabet $S$. The Cayley graph of $(W,S)$ with this orientation is the Hasse diagram of the {\em (right) weak order} $\leq$, see~\cite[Chapter~3]{bjoerner_combinatorics_2005} for definitions and background. 

The weak order encodes a good deal of the combinatorics of reduced words associated to $W$. For instance for  $u,v\in W$,   $u\leq v$ if and only if a reduced word for~ $u$ is a prefix for a reduced word for $v$.  Moreover, Bj\"orner~\cite[Theorem~8]{bjoerner_ordering_1984} shows that the poset $(W,\leq)$ is a complete meet semilattice: for any $A\subseteq W$, there exists an infimum $\bigwedge A\in W$, also called the {\em meet} of $A$, see also~\cite[Theorem 3.2.1]{bjoerner_combinatorics_2005}.  This means, in particular, that any $u,v\in W$ have a common greatest prefix, that is, the unique longest $g=u\wedge v\in W$ for which a reduced word is the prefix of a reduced word for $u$ and of a reduced word for $v$.

In the case of finite Coxeter systems,  the weak order turns out to be a complete ortholattice~\cite[Theorem~8]{bjoerner_ordering_1984}, thanks to the existence of the unique longest element $w_\circ$ in $W$: $u\leq v$ if and only if $uw_\circ\geq vw_\circ$ and the supremum of $A\subseteq W$ exists and is $\bigvee A= (\bigwedge (Aw_\circ))   w_\circ$, also called the {\em join of $A$}.  In particular, for any $u,v\in W$, there exists a unique smallest $g=u\vee v\in W$ with two reduced words, one with prefix a reduced word for $u$ and the other with prefix a reduced word for $v$.  The existence of $w_\circ$ and the fact that the weak order is a lattice play important roles in the study of structures related to $(W,S)$ such as Cambrian lattices and cluster combinatorics~\cite{reading_cambrian_2006,reading_clusters_2007,reading_sortable_2007},  Garside elements in spherical Artin-braid groups, see for instance~\cite{dehornoy_foundations_2015}, or reflection orders, their initial sections and Kazhdan-Lusztig polynomials, see the discussion in~\cite[\S2]{dyer_weak_2011}. 

When trying to generalize this technology to infinite Coxeter systems, the absence of $w_\circ$ and of a join in general is crucially missed. For instance the Coxeter sortable elements and Cambrian fans fail to recover the whole cluster combinatorics~\cite[\S1.2]{reading_sortable_2011}. Examples suggest that we need to adjoin new elements to the weak order of $(W,S)$ in order to generalize properly its combinatorial usage to infinite Coxeter groups. In other words, we need to define a larger family of objects, containing the elements of the infinite Coxeter groups, that would further have greatest common prefixes (meet) and least common multiples (join). This brings us to the following question: is there, for each infinite Coxeter system, a suitable complete ortholattice with the usual weak order as a subposet?

 As a first natural candidate to consider, one could think of the set of finite and infinite reduced words over $S$, modulo finite and infinite sequences of relations, but this fails already in the case of infinite dihedral groups since the two infinite words do not have a join in that case. Further, a completion was obtained by Dyer~\cite[Corollary~10.8]{dyer_groupoid_2011} using the theory of rootoids, but it is not suitable to our combinatorial need.
Hereafter, we investigate an extension of the weak order, containing the set of infinite reduced words, proposed by M.~Dyer and conjectured to be a complete ortholattice~\cite[Conjecture~2.5]{dyer_weak_2011}.

This extension uses the notion of {\em biclosed subsets} of positive roots, a suitable generalizations of inversion sets of words, which we recall in \S\ref{se:1}. This conjecture is still open. The main difficulties are the following:

\begin{enumerate}
	\item understand biclosed sets in general. Finite biclosed sets are the inversion sets of the elements of $W$, see \S\ref{se:1}, but what about the other biclosed sets? 
	\item understand the possible candidates for a join.  
\end{enumerate}

In this article we investigate these two points. After surveying what is known about the notions of biclosed sets and inversion sets in~\S\ref{se:1}, we give in \S\ref{sse:GeomCharac} a geometric criterion for the existence of a join in $(W,\leq)$: the join of a subset $X\subseteq W$ exists if and only if $X$ is finite and the cone spanned by the inversion sets of the elements in $X$ is strictly separated from the {\em imaginary cone} --~the conic hull of the limit roots~--, see Theorem~\ref{thm:Main1}. As a corollary, we obtain that  a subset $A$ of positive roots is a finite biclosed set if and only if $A$ is the set of roots contained in a closed halfspace that does not intersect the imaginary cone (Corollary~\ref{cor:CaracFinite}). The proof is based on a characterization of the existence of the join obtained by Dyer~\cite{dyer_weak_2011} and on the study of limit roots and imaginary cones started in~\cite{hohlweg_asymptotical_2014,dyer_imaginary_2013,dyer_imaginary2_2013}. In \S\ref{se:Inf}, we  use our geometric criterion on finite biclosed sets to give a (partially conjectural) characterization of biclosed sets corresponding to the inversion sets of infinite words, see Corollary~\ref{prop:ImN} and Conjecture~\ref{conj:CharInf}. This characterization extends to arbitrary Coxeter systems a result of P.~Cellini and P.~Papi~\cite{cellini_structure_1998} valid for affine Coxeter systems (see Remark~\ref{rem:CePa}). Finally, in \S\ref{se:OnDef} we discuss the similarities and especially the {\em differences} between the notions of {\em biclosed, biconvex} and {\em separable} subsets of positive roots. The notion of biconvex sets in particular was used in recent works~\cite{baumann_affine_2014,lam_total_2013} related to the affine cases in which the authors also conjectured an extension of the weak order similar to Dyer's conjecture.

\section{Biclosed sets, inversion sets and join in the weak order}\label{se:1}

\subsection{Geometric representation of a Coxeter system} Let $(V,B)$ be a quadratic space: $V$ is a  finite-dimensional real vector space endowed with a symmetric bilinear form $B$. The group of linear maps that preserves $B$ is denoted by $O_B(V)$. The {\em isotropic cone of $(V,B)$} is $Q=\{v\in V\,|\, B(v,v)=0\}$. To any non-isotropic vector $\alpha\in V\setminus Q$,  we associate the $B$-reflection $s_\alpha\in O_B(V)$ defined    by $s_\alpha(v)=v-2 \frac{B(\alpha,v)}{B(\alpha,\alpha)}\alpha$.
\smallskip

A {\em geometric representation of $(W,S)$} is a representation of $W$ as a subgroup of $O_B(V)$ such that $S$ is mapped into a set of $B$-reflections associated to a simple system  $\Delta=\{\alpha_s\,|\,s\in S\}$ ($s=s_{\alpha_s}$).  In this article, we always assume that $S$ is a finite set. Recall that a simple system in $(V,B)$ is a finite subset $\Delta$ in $V$ such that:
\begin{enumerate}[(i)]
 \item $\Delta$ is positively linearly independent:  if $\sum_{\alpha\in \Delta} a_\alpha\alpha$ with $a_\alpha\geq 0$, then all $a_\alpha=0$;
\item for all $\alpha, \beta \in \Delta$ distinct, 
  $\displaystyle{B(\alpha,\beta) \in \ ]-\infty,-1] \cup
    \{-\cos\left(\frac{\pi}{k}\right), k\in \mathbb Z_{\geq 2} \} }$;
\item for all $\alpha \in \Delta$, $B(\alpha,\alpha)=1$.
\end{enumerate}
 Note that, since $\Delta$ is positively linearly independant, the cone $\cone(\Delta)$ is pointed: $\cone(\Delta)\cap\cone(-\Delta)=\{0\}$ (here $\cone(A)$ is the set of non-negative linear combinations of vectors in $A$).  Note  also  that if the order $m_{st}$ of $st$ is  finite, then $B(\alpha_s,\alpha_t)=-\cos\left(\frac{\pi}{m_{st}}\right)$ and that  $B(\alpha_s,\alpha_t)\leq -1$ if and only if the order of $st$ is infinite.  In the case where $a_{st}=B(\alpha_s,\alpha_t)< -1$, we label the corresponding edge in the Coxeter graph by $a_{st}$ instead of $\infty$. 
 
A geometric representation of a Coxeter system is always faithful and gives rise to a root system $\Phi=W(\Delta)$ ($\Phi$ is the orbit of $\Delta$ under the action of $W$), which is partitioned into positive roots $\Phip=\cone(\Delta)\cap \Phi$ and negative roots $\Phi^-=-\Phip$.

The {\em rank of the root system} is the cardinality $|\Delta|=|S|$ of $\Delta$. The {\em dimension of the geometric representation} is the dimension of the linear span of $\Phi$. The {\em classical geometric representation} is obtained by assuming that $\Delta$ is a basis of $V$ and that $B(\alpha_s,\alpha_t)=-1$ if the order of $st$ in $W$ is infinite. Moreover, the combinatorial features of the root system,  such as the inversion sets defined in \S\ref{sse:Inv} below,  do not  depend on the choice of the geometric representation. For more details on geometric representations, we refer  the reader to~\cite[\S3]{bonnafe-dyer} or \cite[\S1 and \S5.3]{hohlweg_asymptotical_2014}.

\subsection{Inversion sets and the weak order}\label{sse:Inv}

We recall here a useful geometric interpretation of the weak order. First, we review the vocabulary concerning the combinatorics of reduced words.  Recall  that, for a word $w$ on an alphabet $S$, a {\em prefix} (resp. suffix) of $w$ is a word  $u$ (resp. $v$) on $S$ such that there is a word $v$ (resp. $u$) on $S$ with $w=uv$ (for the concatenation of words). For $u,v,w\in W$, we say that 

\begin{itemize}
\item {\em $w=uv$ is reduced} if $\ell(w)=\ell(u)+\ell(v)$; 
\item {\em $u$ is a prefix of $w$} if a reduced word for $u$ is a prefix of a reduced word for~$w$;
\item {\em $v$ is a suffix of $w$} if a reduced word for $v$ is a suffix of a reduced word for~$w$.
\end{itemize}
Observe that if $w=uv$ is reduced then the concatenation of any reduced word for~$u$ with any reduced word for~$v$ is a reduced word for $w$. Also $u\leq w$ in the right weak order if and only if $u$ is a prefix of $w$. Similarly, $v\leq_L w$ in the left weak order if and only if $v$ is a suffix of $w$. In this article, we consider only the right weak order and so any mention of the weak order denotes the \emph{right} weak order.

The {\em (left) inversion set of $w\in W$} is $N(w):=\Phip\cap w(\Phi^-)$, and its cardinality is $\ell(w)=|N(w)|$.  Left inversion sets have a strong connection to the weak order  as stated in the following well-known proposition.

\begin{prop}\label{prop:weakRoot} Let $u,v,w\in W$.
\begin{enumerate}[(i)]

\item If $w=s_1\cdots s_k$  is a reduced word for $w$, then
$$
N(w)=\{\alpha_{s_1},s_1(\alpha_{s_2}),\cdots , s_1\dots s_{k-1}(\alpha_{s_k})\}.
$$
\item If $w=uv$ is reduced  then $N(w)=N(u)\sqcup u(N(v))$.

\item The map $N$ is a poset monomorphism from $(W,\leq)$ to $(\mathcal P(\Phip),\subseteq )$. 

\end{enumerate}
\end{prop}
\begin{proof} The first item follows easily from the geometric interpretation of the length function, see for instance~\cite[\S5.6]{humphreys_reflection_1992} for the classical geometric representation and \cite[Lemma 3.1]{bonnafe-dyer} in general. The second item is a direct consequence of the first one by considering a reduced word of $w$ obtained by the concatenation of a reduced word for $u$ and a reduced word for $h$. The fact that $N$ is a poset morphism follows from (ii) and the definitions above. Finally, to show\footnote{The fact that $N$ is injective is classical, see for instance~\cite[Exercise~16, p.225]{bourbaki_groupes_1968} in the case of Weyl groups or \cite[Proposition 3.2]{fleischmann_pointwise_2002} in general, we give a proof here for completeness.} that $N$ is injective, consider $u,v$ such that $N(u)=N(v)$. By definition  of $N$ we have $\Phip\cap u(\Phi^-)=  \Phip\cap v(\Phi^-)$. Then  $\Phi^-\cap u(\Phip)=  \Phi^-\cap v(\Phip)$. By taking complements of $N(u)=N(v)$ in $\Phip$, we obtain $\Phip\cap u(\Phip)=  \Phip\cap v(\Phip)$. It follows that $u(\Phip)=v(\Phip)$. In other words, $N(v^{-1}u)=\varnothing$ and therefore $v^{-1}u=e$, which concludes the proof.
\end{proof}

\subsection{Biclosed sets and the weak order} We shall now describe the image of the map $N$.  A subset $A\subseteq \Phip$ is {\em closed} if $\cone(\alpha,\beta)\cap\Phi\subseteq A$, for all $\alpha,\beta\in A$, and $A$ is {\em coclosed} if $A^c=\Phip\setminus A$ is closed. A subset of $\Phip$ is {\em biclosed} if it is both closed and coclosed. Observe that the intersection of closed sets is closed and that the union of coclosed sets is coclosed. Let $\mathcal B(\Phip)=\mathcal B$ denote the set of biclosed sets and $\mathcal B_\circ(\Phip)=\mathcal B_\circ$ the set of finite biclosed sets.   The following proposition can be found in~\cite[Lemma 4.1 (d) and (f)]{dyer_weak_2011} for instance. 

\begin{prop}\label{prop:biclosed}
The map $N:(W,\leq)\to (\mathcal B_\circ,\subseteq )$ is a poset isomorphism.
\end{prop}

Examples are given in Figure~\ref{fig:A2inv} and Figure~\ref{fig:InfiniteDi}
\begin{figure}[!htbp]
	\begin{center}
		\begin{tikzpicture}
	[scale=1,
	 pointille/.style={dashed},
	 axe/.style={color=black, very thick},
	 sommet/.style={inner sep=2pt,circle,draw=blue!75!black,fill=blue!40,thick,anchor=west}]
	 
   \node[sommet]  (id)    [label=below:${N(e)=\varnothing}$]          at (0,0)    {};
	\node[sommet]  (e1)    [label=left:${N(s_1)=\{\alpha_{s_1}\}}$]         at (-1,1)   {} edge[thick] (id);
	\node[sommet]  (e2)    [label=right:${\{\alpha_{s_2}\}=N(s_2)}$]        at (1,1)    {} edge[thick] (id);
	\node[sommet]  (e12)   [label=left:${N(s_1s_2)=\{\alpha_{s_1},\alpha_{s_1}+\alpha_{s_2}\}}$]      at (-1,2)   {} edge[thick] (e1);
	\node[sommet]  (e21)   [label=right:${\{\alpha_{s_2},\alpha_{s_1}+\alpha_{s_2}\}=N(s_2s_1)}$]     at (1,2)    {} edge[thick] (e2);
	\node[sommet]  (e121)  [label=above:${N(s_1s_2s_1)=\Phip}$]    at (0,3)    {} edge[thick] (e12) edge[thick] (e21);
\end{tikzpicture}
		\caption{\label{fig:A2inv} The weak order on the Coxeter group of type $A_2$.}
	\end{center}
\end{figure}
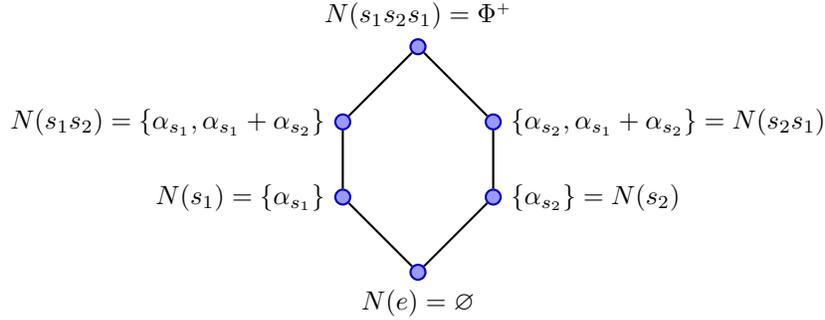

\begin{remark} As far as we know, this proposition was stated first in an article by Kostant~\cite[Proposition 5.10]{kostant_lie_1961}, then as an exercise in Bourbaki~\cite[p.225, Exercice 16]{bourbaki_groupes_1968}, in the case of finite crystallographic root system. An extended version for a possible infinite non crystallographic root system in the classical geometric representation was stated in \cite[Proposition~2 and Proposition~3]{bjoerner_ordering_1984}  for biconvexity instead of bicloseness. A version involving initial section can be found in~\cite[\S2.11]{dyer_hecke_1993}.  Then the first  complete proof stated for finite biclosed sets in a infinite Coxeter group can be found in \cite[\S8]{pilkington_convex_2006}.  The proof for an arbitrary root system in an arbitrary geometric representation is precisely the same as the ones given before and can be found in~\cite[Lemma 4.1]{dyer_weak_2011}. 
\end{remark}

\subsection{Join in the weak order} The weak order is a complete meet semi-lattice~\cite[Theorem~3.2.1]{bjoerner_combinatorics_2005}: any non-empty subset $X\subseteq W$ admits a greatest lower bound called the {\em meet} of $X$ and denoted by $\bigwedge X$. As explained in the introduction, the weak order turns out to be a complete ortholattice whenever $W$ is finite. So in this case $X$ always admits a least upper bound called the {\em join} of~$X$, denoted by $\bigvee X$. When $W$ is infinite, however, it is not immediately clear if and when the join of $X$ exists and how to compute it.   From the point of view of combinatorics of reduced words, the join of $X$, if it exists, should be the smallest length element in $W$ that has any element in $X$ as a prefix. 

\begin{remark} In general, $N(\bigwedge X)\not = \bigcap_{x\in X} N(x)$. For instance, let $(W,S)$ be of type $A_2$ and with $S=\{s_1,s_2\}$ as in Figure~\ref{fig:A2inv}.   The inversion set $$N(s_1s_2 \wedge s_2s_1)=N(e)=\varnothing$$ is not equal to $N(s_1s_2) \cap N(s_2s_1)=\{\alpha_{s_1}+\alpha_{s_2}\}$, which is not even biclosed. 
\end{remark}

In~\cite{dyer_weak_2011}, Dyer studied the problem  of the existence of join in the weak order. We explain now the characterization he obtains. 

\begin{defi} $\ $
\begin{enumerate}
\item  We say that $X$ is {\em bounded in $W$} if there is $g\in W$ such that $x\leq g$ for all $x\in X$.  Observe that $X$ is necessarily finite since there are only a finite number of prefixes of $g$.
 
\item Similarly, we say that $A\subseteq \Phip$ is {\em bounded in $\Phip$} if there is a finite biclosed set $B\in \mathcal B_\circ$ containing $A$. Hence $A$ is therefore finite.

\item The {\em $2$-closure} $\overline{A}$ of $A\subseteq \Phip$ is the intersection of all closed subsets of $\Phip$ containing $A$. Since the intersection of closed sets is closed, $\overline{A}$ is the smallest closed subset of $\Phip$ containing $A$. 
\end{enumerate}
\end{defi}

The $2$-closure terminology was introduced by Dyer~\cite{dyer_weak_2011}; see also~\cite{pilkington_convex_2006} for a discussion on different types of closure operators.

\begin{ex} Assume that $W$ is the infinite dihedral group generated with Coxeter graph:
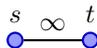
\begin{center}
\begin{tikzpicture}[sommet/.style={inner sep=2pt,circle,draw=blue!75!black,fill=blue!40,thick}]
	\node[sommet,label=above:$s$] (alpha) at (0,0) {};
	\node[sommet,label=above:$t$] (beta) at (1,0) {} edge[thick] node[auto,swap] {$\infty$} (alpha);

\end{tikzpicture}
\end{center}
Then $X$ is bounded if and only if $X$ is finite and all $x\in X$ start with the same letter (either $s$ or $t$), see~Figure~\ref{fig:InfiniteDi}.
\end{ex}

By abuse of notation, we write  $N(X):=\bigcup_{x\in X}N(x)$.

\begin{thm}[{Dyer~\cite[Theorem~1.5]{dyer_weak_2011}}]\label{thm:Dyer} Let $X\subseteq W$. Then  the following statements are equivalent: 

\begin{itemize}
	\item[(i)] $\bigvee X$ exists;
	\item[(ii)] $X$ is bounded in $W$;
	\item[(iii)] $N(X)$ is bounded in $\Phip$.
\end{itemize}
\noindent Moreover, in this case, $N(\bigvee X)= \overline{N(X)}$. 
\end{thm}

We discuss briefly some consequences and questions raised by this theorem.

\begin{enumerate}[(a)]

\item  Theorem~\ref{thm:Dyer} gives necessary and sufficient conditions for the join to exist and a way to compute it: if $X$ is bounded then $\overline{N(X)}$ is finite and is the inversion set of the join $\bigvee X$. Unfortunately, the computation of the $2$-closure of a set of positive roots is not necessarily easy and we do not know of any combinatorial rule to produce $\bigvee X$, not even if $X=\{u,v\}$.  However, we explain in the next part of this section how to compute the join geometrically, if it exists.  

\item The theorem states also that if $\bigvee X$ exists, then $\overline{N(X)}$ is a finite biclosed set since it is the inversion set of an element, by Proposition~\ref{prop:biclosed}. Dyer~\cite[Remarks 1.5]{dyer_weak_2011} asks if the converse of this last statement is true  if we drop the biclosed condition. This is discussed in \S\ref{sse:GeomCharac}.  

\item If $W$ is finite, then $\mathcal B=\mathcal B_\circ$ and the poset $(\mathcal B,\subseteq)$ is a complete ortholattice isomorphic to the weak order with maximal element $N(w_\circ)=\Phip$; the join is given by the formula given in Theorem~\ref{thm:Dyer} and the meet is $\bigwedge A= \overline{A^c}^c$ for $A\in \mathcal B$.  Dyer conjectures that this is true even if $W$ is infinite.

\begin{conject}[{Dyer~\cite[Conjecture 2.5]{dyer_weak_2011}}]\label{conj:Dyer}  The poset $(\mathcal B,\subseteq)$ of biclosed sets ordered by inclusion is a complete ortholattice. The join of a family $A \subseteq\mathcal B$, is $\bigvee A=\overline{\bigcup A}$, and the ortholattice complement is the set complement in $\Phip$.
\end{conject}

In other words, for this conjecture to be true, we have to show that the closed set $\overline{\bigcup A}$ is also coclosed; this is the case for infinite dihedral groups as readily seen in Figure~\ref{fig:InfiniteDi}. The present work was motivated by this  conjecture.  We believe that our present investigation of biclosed sets and joins in the weak order will lead to an understanding of the existence of $\bigvee A$ in the case where the family~$A$ in Conjecture~\ref{conj:Dyer} is constituted of finite biclosed sets and biclosed sets corresponding to infinite reduced words. This would be a first step toward answering this conjecture.

\begin{figure}[h!]
\resizebox{.95\hsize}{!}{
\begin{tikzpicture}
	[scale=1,
	 pointille/.style={dashed},
	 axe/.style={color=black, very thick},
	 sommet/.style={inner sep=2pt,circle,draw=blue!75!black,fill=blue!40,thick,anchor=west}]
	 
\node[sommet]  (id)    [label=below:{\small{$N(e)=\varnothing$}}]          at (0,0)    {};
\node[sommet]  (1)    [label=left:{\small{$N(s)=\{\alpha_s\}$}}]         at (-1,1)   {} edge[thick] (id);
\node[sommet]  (2)    [label=right:{\small{$N(t)=\{\alpha_t\}$}}]        at (1,1)    {} edge[thick] (id);
\node[sommet]  (12)   [label=left:{\small{$N(st)=\{\alpha_s,s(\alpha_t)\}$}}]      at (-1,2)   {} edge[thick] (1);
\node[sommet]  (21)   [label=right:{\small{$N(ts)=\{\alpha_t,t(\alpha_s)\}$}}]     at (1,2)    {} edge[thick] (2);
\node[sommet]  (121)  [label=left:{\small{$N(sts)=\{\alpha_s,s(\alpha_t),st(\alpha_s)\}$}}]    at (-1,3)    {} edge[thick] (12);
\node[sommet]  (212)  [label=right:{\small{$N(tst)=\{\alpha_t,t(\alpha_s),ts(\alpha_t)\}$}}]    at (1,3)    {} edge[thick] (21);

\draw[pointille] (121) -- +(0,1);
\draw[pointille] (212) -- +(0,1);

\node[sommet]  (12inf)  [label=left:{\small{$N((st)^\infty)=\{(n+1)\alpha_s+n\alpha_t\,|\,n\in\mathbb N\}$}}]    at (-1,4.5)    {};
\node[sommet]  (21inf)  [label=right:{\small{$N((ts)^\infty)=\{n\alpha_s+(n+1)\alpha_t\,|\,n\in\mathbb N\}$}}]    at (1,4.5)    {};

\node[sommet]  (12c)   [label=left:{\small{$\Phip\setminus N(ts)$}}]      at (-1,6)   {} ;
\node[sommet]  (21c)   [label=right:{\small{$\Phip\setminus N(st)$}}]     at (1,6)    {} ;
\draw[pointille] (12c) -- +(0,-1);
\draw[pointille] (21c) -- +(0,-1);
\node[sommet]  (1c)    [label=left:{\small{$\Phip\setminus N(t)$}}]         at (-1,7)   {} edge[thick] (12c);
\node[sommet]  (2c)    [label=right:{\small{$\Phip\setminus N(s)$}}]        at (1,7)    {} edge[thick] (21c);
\node[sommet]  (id)    [label=above:{\small{$\Phip$}}]          at (0,8)    {} edge[thick] (1c) edge[thick] (2c);

\end{tikzpicture}}
\caption{The weak order on biclosed sets for the infinite dihedral group.}
\label{fig:InfiniteDi}
\end{figure}
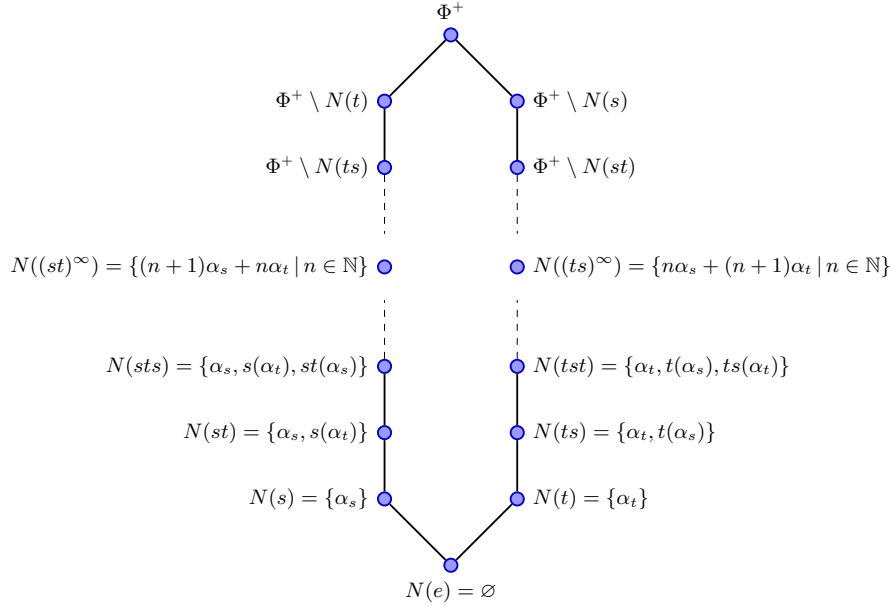

\begin{remark} Conjecture~\ref{conj:Dyer} is related to the theory of initial sections and reflection orders, developed by Dyer\cite{dyer_hecke_1993,dyer_quotients_1994} as a tool to study Kazhdan-Lusztig polynomials, see~\cite[Introduction and \S2]{dyer_weak_2011}. Initial sections and reflection orders appeared also in the works of Cellini-Papi~\cite{cellini_structure_1998} and Ito~\cite{ito_classification_2001,ito_parameterizations_2005}.  Recently, questions about initial sections and the weak order appeared in the work of Baumann, Kamnitzer and Tingley~\cite[\S2.5]{baumann_affine_2014}.  The same year, Lam and Pylyavskyy~\cite[Theorem~4.10]{lam_total_2013} showed that the set of biclosed sets corresponding to inversion sets of finite and infinite reduced words in an affine Coxeter group is a meet semi-lattice. We discuss this particular class of (possibly infinite) biclosed sets in~\S\ref{se:Inf}. 
\end{remark}

\end{enumerate}
\subsection{Geometric construction of the join in the weak order}

Here we extract from~\cite[\S11 and Theorem 11.6]{dyer_weak_2011} that if the join exists, then the  $2$-closure can be replaced by the conic closure. For $A\subseteq V$, denote $\cone(A)$ the set of all non-negative linear combinations of elements of $A$ and $\cone_\Phi(A)=\cone(A)\cap \Phi$ the (possibly empty) set of roots contained in $\cone(A)$. Observe that $\cone_\Phi(\Delta)=\Phip$. As an immediate consequence, we have that if $A\subseteq\cone(\Delta)$ then $\cone_\Phi(A)\subseteq \Phip$.

\begin{defi} Let $A\subseteq \Phip$.
\begin{enumerate}
\item The set $A$ is {\em convex} if $A=\cone_\Phi(A)$ and {\em coconvex} if $A^c=\Phip\setminus A$ is convex.
\item The set $A$ is {\em biconvex} if $A$ is convex and coconvex.
\item The set $A$ is {\em separable} if $\cone(A)\cap\cone(A^c)=\{0\}$.
\end{enumerate}
\end{defi}

Note that $A$ is biconvex (resp. separable) if and only if $A^c$ is. The next lemma exhibits some relationship between these notions. 

\begin{lem}\label{lem:SepaConvClos} Let $A\subseteq \Phip$.
\begin{enumerate}[(i)]
\item If $A$ is convex, then $A$ is closed.
\item $A\textrm{ separable}\Longrightarrow A\textrm{ biconvex}\Longrightarrow A\textrm{ biclosed}$.
\item For any linear hyperplane $H$, the set of positive roots contained (strictly or not) in one side of the halfspace bounded by $H$  is separable (hence biclosed).
\end{enumerate}
\end{lem}
\begin{proof} The first item follows easily from the definitions. The second implication of~$(ii)$ follows from $(i)$. Assume that $A$ is separable. Suppose by contradiction that $\cone_\Phi(A)\not = A$.  Note that $A\subseteq \cone_{\Phi}(A)\subseteq \cone(A)$. So there is $\beta\in A^C\cap \cone_\Phi(A)$. Hence $\beta\in\cone(A)\cap \cone(A^c)$,  contradicting the definition of separable.  Thus $A$ is convex. The proof is similar to show that $A$ is coconvex. Therefore $A$ is biconvex. 

Now let us prove $(iii)$. Consider a hyperplane $H=\ker\rho$ ($\rho$ is a linear form on $V$) and $A:=\{\beta\in\Phip\,|\, \rho(\beta)\geq 0\}$. Write $H_{\geq 0}=\{v\in V\,|\,\rho(v)\geq 0\}$ and $H_{<0}=\{v\in V\,|\,\rho(v)<0\}$. Note that $A^c\subseteq H_{<0}$. Therefore we have $\cone(A)\subseteq H_{\geq 0}$ and $\cone(A^c)\setminus\{0\}\subseteq H_{<0}$. So those two cones can only intersect in $0$, which implies that $A$ is separable. 
\end{proof}

The converse of $(i)$ is false, a counterexample lives in the finite Coxeter group of type $D_4$, see \cite[p.3192]{pilkington_convex_2006}.  If $A$ is infinite, the converse of any of the implications in $(ii)$ is not true nor is the converse of $(iii)$, as it is shown in \S\ref{se:OnDef}.
However, those equivalences  holds if $A$ is finite as stated in the following  proposition.

\begin{prop}\label{prop:FiniteBiclos} Let $A\subseteq \Phip$ be finite. Then the following assertions are equivalent:
\begin{enumerate}[(i)]
\item $A=N(w)$ for some $w\in W$;
\item  $A$ is biclosed;
\item $A$ is biconvex;
\item $A$ is separable;
\item There exists a hyperplane $H$ such that $A$ is strictly on one side of $H$ and $A^c$ is strictly on the other side of $H$.
\end{enumerate}
\end{prop}

This proposition is a reformulation of~\cite[Proposition 11.6]{dyer_weak_2011} as well as  the beginning of its proof. We give here a proof for completeness. For further discussions about closed and convex sets on different types of finite root systems, see~\cite{pilkington_convex_2006}. 

\begin{proof} By Proposition~\ref{prop:biclosed} and Lemma~\ref{lem:SepaConvClos} we only have to show that $(i)$ implies~{{$(v)$}}. Take $w\in W$ and write $A=N(w)$.  Note that $\Delta$ is finite so $\cone(\Delta)$ is a simplicial convex cone. Since $\Delta$ is  a simple system, we have $\cone(\Delta)\cap-\cone(\Delta)=\{0\}$.  Observe that $\cone(\Delta)$ is pointed at $0$, so a supporting hyperplane  $H_0$ of the face $0$ separates strictly $\cone(\Delta)\setminus\{0\}$ and $-\cone(\Delta)\setminus\{0\}$ (see for instance~\cite{webster_convexity_1994,ziegler_lectures_1995} for further information on cones and polytopes).

Therefore  the sets~$\Phip$  and $\Phi^-$ are  strictly separated by  $H_0$, since $\Phi^\pm=\pm\cone_\Phi(\Delta)$. Therefore the hyperplane $H=w(H_0)$ strictly separates  $A=N(w)$ and $A^c$ since $N(w)\subseteq w(\Phi^-)$ and $\Phip\setminus N(w)\subseteq w(\Phip)$. 
 \end{proof}

\begin{ex}
	Let $(W,S)=(A_2,\{s_1,s_2\})$. Figure~\ref{fig:example_geometry} shows the corresponding root system from which the biclosed, biconvex and separable sets can be easily obtained.
	\begin{figure}[!htbp]
		\begin{center}
			\begin{tikzpicture}
				[scale=1,
				pointille/.style={dashed},
				axe/.style={color=black, very thick}]
				
				\coordinate (O) at (0,0);
				\fill (O) circle (0.05);

				\draw[axe,->] (O) -- (0:1) node[label=right:{$\alpha_{s_1}$}] {};
				\draw[axe,->] (O) -- (60:1) node[label=above right:{$\alpha_{s_1}+\alpha_{s_2}$}] {};
				\draw[axe,->] (O) -- (120:1) node[label=above left:{$\alpha_{s_2}$}] {};
				\draw[axe,->] (O) -- (180:1) node[label=left:{$-\alpha_{s_1}$}] {};
				\draw[axe,->] (O) -- (240:1) node[label=below left:{$-\alpha_{s_1}-\alpha_{s_2}$}] {};
				\draw[axe,->] (O) -- (300:1) node[label=below:{$-\alpha_{s_2}$}] {};
				
				\draw[dotted] (150:1.5) -- (-30:1.5) node[label=right:{$H$}] {};
				
			\end{tikzpicture}
		\end{center}
		\caption[Root system of type $A_2$]{\label{fig:example_geometry} The root system of type $A_2$. The inversion sets $\varnothing,$ $\{\alpha_{s_1}\},$ $\{\alpha_{s_2}\},$ $\{\alpha_{s_1},\alpha_{s_1}+\alpha_{s_2}\},$ $\{\alpha_{s_2},\alpha_{s_1}+\alpha_{s_2}\}$ and $\{\alpha_{s_1},\alpha_{s_2},\alpha_{s_1}+\alpha_{s_2}\}$ are the biclosed, biconvex and separable sets of $\Phip$.}
	\end{figure}
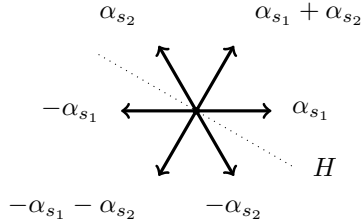
\end{ex}

The following corollary is  due to Dyer~\cite[Proposition 11.6]{dyer_weak_2011}.

\begin{cor}\label{cor:JoinFinite}  Let $A\subseteq \Phip$ such that $\overline{A}$ is a finite biclosed set. Then
$$\overline{A}= \cone_\Phi(A).$$
\end{cor}
\begin{proof} On one hand $A':=\cone_\Phi(A)$ is convex by definition, hence closed by Lemma~\ref{lem:SepaConvClos}. So by minimality $\overline{A}\subseteq A'$.  On the other hand, we have $A\subseteq \overline{A}$. Thus
$
\cone(A)\subseteq \cone(\overline{A}).
$
Since $\overline{A}$ is finite biclosed, it is biconvex by Proposition~\ref{prop:FiniteBiclos}. So $\cone_\Phi(\overline{A})=\overline{A}$ and $A'\subseteq \overline{A}$.
\end{proof}

\begin{prop}\label{prop:JoinFinite} Let $X$ be a bounded subset of $W$. Then the join $\bigvee X$ exists and 
$$
N(\bigvee X)= \overline{N(X)}=\cone_\Phi(N(X)).
$$
\end{prop}
\begin{proof} This is an immediate consequence of Theorem~\ref{thm:Dyer} and Corollary~\ref{cor:JoinFinite}.
\end{proof}

\begin{remark} 
\begin{enumerate}
\item Such a formula to compute the join in finite Coxeter groups was originally proven by Bj\"orner, Edelman and Ziegler in \cite[Theorem~5.5]{BjEdZi90}, as a special case of their study of `Hyperplane arrangements with a lattice of regions'. It was then extended to infinite Coxeter groups by Dyer in~\cite[Proposition 11.6]{dyer_weak_2011}. 
\item The second author (JPL) shows in his thesis~\cite[Corollary~2.36]{labbe_polyhedral_2013} that Conjecture~\ref{conj:Dyer} is true for rank $3$ Coxeter system (i.e. $|S|=3$) if we consider biconvex sets instead of biclosed sets. In this case, the join of two  biconvex sets is given by the conic closure, as in Proposition~\ref{prop:JoinFinite}. The authors tried in vain to prove or disprove that, for Coxeter groups of rank $3$ or affine Coxeter groups, biclosed sets are all biconvex sets and that the $2$-closure can be replaced by the conic closure in Conjecture~\ref{conj:Dyer}.  However, this is no longer true if the Coxeter group is indefinite (not finite nor affine) of rank~$\geq 4$ as  shown in~\cite[Section~2.4.2]{labbe_polyhedral_2013}.
\end{enumerate}
\end{remark}

\section{Existence of the join and the imaginary cone} \label{sse:GeomCharac}

\subsection{Projective representation and normalized roots}\label{sse:Normal} \smallskip In this article we also refer to the  {\em projective representation} of $W$ and to the associated {\em normalized root system} $\h \Phi$. Since the root system is encoded by the set of positive roots $\Phip$, we represent the root system by an ``affine cut'' of $\Phip$. This projective representation has nice consequences for the study of infinite Coxeter groups, as explained in \cite{hohlweg_asymptotical_2014,dyer_imaginary2_2013}. It is especially useful for easily visualizing infinite root systems and to work out examples of rank $|S|=2,3,4$ easily. It works as follows: there is an affine hyperplane $V_1$ in $V$ {\em transverse to $\Phip$}, i.e., for any $\beta\in \Phip$, the ray $\mathbb R^+\beta$ intersects $V_1$ in a unique nonzero point $\h\beta$. So $\mathbb R\beta\cap V_1=\{\h\beta\}$ for any $\beta\in\Phi$. The {\em set of normalized roots}, which is a projective view of $\Phi$, is
$$
\h\Phi=\{\h\beta\,|\, \beta\in \Phi\},
$$
see for instance~\cite[Figures~2~and~3]{hohlweg_asymptotical_2014}. Observe that $\h\Phi$ is contained in the polytope $\conv(\h\Delta)$. If $W$ is infinite, so is $\h\Phi$, and thus $\h\Phi$ admits a set $E$ of accumulation points that we call {\em the set of limit roots}, see \cite[Figures 4--7]{hohlweg_asymptotical_2014}\cite[Appendix~A]{labbe_polyhedral_2013}\cite{chen_limit_2014}.

Assuming $(W,S)$ to be irreducible, we have   $E=\varnothing$ if and only if $W$ is finite; $E$ is a singleton if and only if $(W,S)$ is affine and irreducible. Moreover, limit roots are in the isotropic cone $Q$ of $B$:
$$
E\subseteq \h Q=\{x\in V_1\,|\, B(x,x)=0\}.
$$
A nice observation is that for any roots $\alpha,\beta\in\Phip$, the dihedral reflection subgroup generated by $s_\alpha,s_\beta$ is finite if and only if the line  $L(\h\alpha,\h\beta)$ through $\h\alpha$ and $\h\beta$ is such that  $L(\h\alpha,\h\beta)\cap Q=\varnothing$. Otherwise $L(\h\alpha,\h\beta)$ intersects $Q$ in one or two points and contains an infinite number of normalized roots. 

In this framework, conic closure is replaced by convex hull and $A\subseteq \Phip$ is replaced by $\h A \subseteq \h\Phi$. For instance, this gives
for $A\subseteq \Phip$.
\begin{itemize}
\item  The set $A$ is {\em closed} if for any $\alpha,\beta\in A$, the normalized roots in the segment $[\h\alpha,\h\beta]$ are all contained in $\h A$.
\item The set $A$ is {\em convex} if $\h A=\conv_\Phi(\h A)=\conv(\h A)\cap\h\Phi$.
\item The set $A$ is {\em separable} if $\conv(\h A)\cap\conv(\h A^c)=\varnothing$.
\end{itemize}

We write $\h N (w)$ instead  of $\h{N(w)}$ and we refer to the map $\h N:W\to \h{\mathcal B}$,  where $\h{\mathcal B}$ is the set of normalized biclosed sets. We omit the notation $\,\h{\cdot}\,$ in figures. Within this setting, Proposition~\ref{prop:JoinFinite} translates as follows (see Figure~\ref{fig:AffineC} for an illustration): 

\begin{prop} Let $X$ be a bounded subset of $W$. Then the join $\bigvee X$ exists and 
$
\h N(\bigvee X)= \conv_\Phi(\h N(X)).
$
\end{prop}

\subsection{A geometric criterion for the existence of the join in $W$}

A very useful feature studied in \cite{dyer_imaginary_2013,dyer_imaginary2_2013} is the {\em imaginary convex body} $\conv(E)$ which is shown to be a compact set~\cite[Theorem 2.2]{dyer_imaginary2_2013}. This lead us to the following result, which is a geometric generalization of Theorem~\ref{thm:Dyer}. Recall that, by abuse of notation, we write  $N(X):=\bigcup_{x\in X}N(x)$.

\begin{thm}\label{thm:Main1} Let $X\subseteq W$. Then the following statements are equivalent:
\begin{enumerate}[(i)]
\item $\bigvee X$ exists;
\item $X$ is bounded;
\item $X$ is finite and $\conv(\h N(X))\cap\conv(E)=\varnothing$. 
\item $X$ is finite and the convex set $\conv(\h N(X))$ and  the imaginary convex body $\conv(E)$ are strictly separated by an hyperplane.   
\end{enumerate}
In this case: $N(\bigvee X)= \cone_\Phi(N(X))$. 
\end{thm}
\begin{proof} The equivalence between $(i)$ and $(ii)$ is Dyer's Theorem~\ref{thm:Dyer}. The equivalence between $(iii)$ and $(iv)$ is a consequence of the Hahn-Banach Separation Theorem. Indeed, since $X$ is finite, $\conv(\h N(X))$ is a polytope, hence compact. Since both convex sets $\conv(E)$ and $\conv(\h N(X))$ are compact, then $\conv(\h N(X))\cap\conv(E)=\varnothing$ if and only if there is an hyperplane $H_1$ in $V_1$ that separates strictly $\conv(E)$ and $\conv(\h N(X))$, by  the Hahn-Banach Separation Theorem; See for instance~\cite[Theorem 2.4.10]{webster_convexity_1994}.

Assume $(i)$: write $A=\cone_\Phi(N(X))$. Since $\bigvee X$ exists, the set $X$ is finite and $N(\bigvee X)= A$ by Proposition~\ref{prop:JoinFinite}. By Proposition~\ref{prop:FiniteBiclos}, there is a hyperplane $H$ separating strictly  $A=\cone_\Phi(N(x))$  and $A^c$. That means, if $H^+$ (resp. $H^-$) denotes the open half-space with boundary $H$ containing $A$ (resp. $A^c$), then $\Phip\cap H^+=A$. We now show that $E\cap \h H^+=\varnothing$. Indeed, otherwise there is $x\in E\cap \h H^+$. Since $\h H^+$ is open, there is a neighbourhood of $x$ in $\h H^+$, which contains an infinite number of normalized roots.  This means that there is an infinite number of roots in $H^+\cap \Phip=A$,  contradicting the fact that $A$ is finite. Thus $E$ is contained in the closed half-space $\h H\cup \h H^-$, so is $\conv(E)$. Since $\conv(\h A)\subseteq \h H^+$, we conclude that  $\conv(\h N(X))\cap\conv(E)=\varnothing$. So this proves that both $(iii)$ and $(iv)$ are true.

Assume $(iv)$ to be true. Let $H_1$ be the affine hyperplane separating strictly $\conv(E)$ and $\conv(\h N(X))$. Let $H$ be the  linear extension of $H_1$ in $V$. Set $H^+$ to be the open half space bounded by $H$ and containing $N(X)$. So $H^-=V\setminus(H\cup H^+)$ is the open half space bounded by~$H$ and containing $\conv(E)$, hence $E$. The set $A:=H^+\cap \Phip$ is therefore biclosed by Lemma~\ref{lem:SepaConvClos}$(iii)$. Assume $A$ is infinite: since $\h A\subseteq \h\Phi$ is discrete, its accumulation points  lie in $(H^+\cup H)\cap E=\varnothing$ a contradiction. So $A$ is a finite biclosed set containing $N(X)$ by definition. So $N(X)$ is bounded in $\Phip$ and therefore $\bigvee X$ exists by Theorem~\ref{thm:Dyer}.

\end{proof}

We now give an example of application of Theorem~\ref{thm:Main1}.  

\begin{ex} Consider the affine Coxeter group of type $\tilde C_2$ given by the following Coxeter graph:
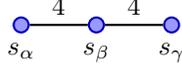
\begin{center}
\begin{tikzpicture}[sommet/.style={inner sep=2pt,circle,draw=blue!75!black,fill=blue!40,thick}]
	\node[sommet,label=below:$s_\alpha$] (alpha) at (0,0) {};
	\node[sommet,label=below:$s_\beta$] (beta) at (1,0) {} edge[thick] node[auto,swap] {4} (alpha);
	\node[sommet,label=below:$s_\gamma$] (gamma) at (2,0) {} edge[thick] node[auto,swap] {4} (beta);
	
\end{tikzpicture}
\end{center}
 We illustrate the normalized root system (with finitely many roots drawn) in Figure~\ref{fig:AffineC}. 
Here $E=\h Q=\{\delta\}$ is a singleton represented by a red dot in the center of Figure~\ref{fig:AffineC}.  Consider $X=\{s_\alpha,s_\gamma s_\beta\}$. Then 
\[
	N(X)=N(s_\alpha)\cup N(s_\gamma s_\beta)=\{\alpha,\gamma,s_\gamma(\beta)\}.
\]
We see that  $\delta\notin \conv_\Phi(\h N(X))$, so by Theorem~\ref{thm:Main1} the join $g=s_\alpha\vee s_\beta s_\gamma$ exists and its  inversion set is equal to $N(g)=\cone_\Phi(\alpha,\gamma,s_\gamma(\beta))$ (the blue triangle at the bottom of  Figure~\ref{fig:AffineC}  that contains exactly 5 roots).  Since $\gamma\in N(g)$, we know that there is $g'\in W$ such that  $g=s_\gamma g'$ is reduced. By  Proposition~\ref{prop:weakRoot}(ii), we have therefore $N(g)=\{\gamma\}\sqcup s_\gamma(N(g'))$;
 in Figure~\ref{fig:AffineC}, the normalized version of the set $s_\gamma(N(g'))$ is constituted of the roots in the blue triangle  but excluding $\gamma$. We represented all roots in this triangle and we see that the remaining normalized roots are on the segment between $\h\alpha$ and $\h{s_\gamma (\beta)}$.  Now, we see that $\alpha$ is in $N(g')$, by proceeding with $g'$ and $\alpha\in N(g')$ the same way as we did above with $\gamma$ and $g$, we obtain recursively that  $g=s_\gamma (s_\alpha s_\beta)^2$.

However, $\delta\in \conv(\h\beta,\h\alpha,\h{s_\beta(\gamma)})$ (the orange triangle at the top of  Figure~\ref{fig:AffineC})  so the join of $s_\alpha$ and $s_\beta s_\gamma$ does not exist.  
 
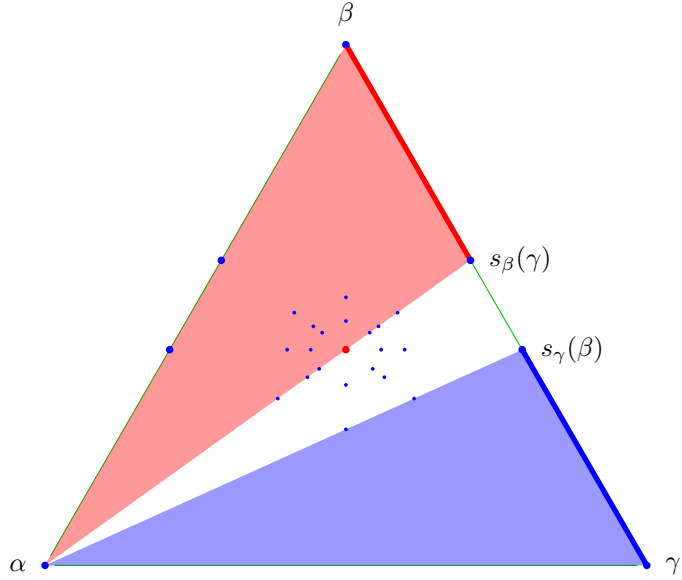
\begin{figure}[h!]
\begin{tikzpicture}
	[scale=2,
	 q/.style={red,line join=round,thick},
	 racine/.style={blue},
	 racinesimple/.style={black},
	 racinedih/.style={blue},
	 limit/.style={fill=red,draw=black,diamond},
	 limdir/.style={fill=orange,draw=black,diamond},
	 weight/.style={fill=green,draw=black,diamond},
	 sommet/.style={inner sep=2pt,circle,draw=black,fill=blue,thick,anchor=base},
	 rotate=0]

\def\grosseur{0.0125}
\def\grosseursimple{0.025}

\def\grosseurdih{0.0075}

\fill[red] (2.00000000000000,1.43487787042860) circle (\grosseursimple);

\node[label=left :{$\alpha$}] (a) at (0.000000000000000,0.000000000000000) {};
\node[label=right :{$\gamma$}] (b) at (4.00000000000000,0.000000000000000) {};
\node[label=above :{$\beta$}] (g) at (2.00000000000000,3.46410161513775) {};
\node[label=right :{${s_\beta(\gamma)}$}] at (2.82842712474619,2.02922374470915) {};
\node[label=right :{${s_\gamma(\beta)}$}] at (3.17157287525381,1.43487787042860) {};

\draw[green!75!black] (a) -- (b) -- (g) -- (a);

\draw[blue,line width = 2pt] (4,0) -- (3.17157287525381,1.43487787042860);
\draw[red,line width = 2pt] (2.00000000000000,3.46410161513775) -- (2.82842712474619,2.02922374470915);

\fill[racine] (0.000000000000000,0.000000000000000) circle (\grosseursimple);
\fill[racine] (4.00000000000000,0.000000000000000) circle (\grosseursimple);
\fill[racine] (2.00000000000000,3.46410161513775) circle (\grosseursimple);

\fill[red,fill opacity=0.4] (0,0) -- (2.82842712474619,2.02922374470915) -- (2.00000000000000,3.46410161513775);
\fill[blue,fill opacity=0.4] (0,0) -- (4,0) -- (3.17157287525381,1.43487787042860);

\fill[racine] (1.17157287525381,2.02922374470915) circle (\grosseursimple);
\fill[racine] (2.82842712474619,2.02922374470915) circle (\grosseursimple);
\fill[racine] (3.17157287525381,1.43487787042860) circle (\grosseursimple);
\fill[racine] (0.828427124746190,1.43487787042860) circle (\grosseursimple);

\fill[racine] (1.54691816067803,1.10981931806051) circle (\grosseur);
\fill[racine] (2.45308183932197,1.10981931806051) circle (\grosseur);
\fill[racine] (2.00000000000000,0.904836765142137) circle (\grosseur);

\fill[racine] (1.65685424949238,1.68106399229610) circle (\grosseur);
\fill[racine] (2.00000000000000,1.78303762284165) circle (\grosseur);
\fill[racine] (2.34314575050762,1.68106399229610) circle (\grosseur);

\fill[racine] (1.60947570824873,1.43487787042860) circle (\grosseur);
\fill[racine] (1.74452083820543,1.25158717262127) circle (\grosseur);
\fill[racine] (2.39052429175127,1.43487787042860) circle (\grosseur);
\fill[racine] (2.25547916179457,1.25158717262127) circle (\grosseur);

\fill[racine] (2.21638837510878,1.59012331585939) circle (\grosseur);
\fill[racine] (1.78361162489122,1.59012331585939) circle (\grosseur);
\fill[racine] (2.00000000000000,1.20047127350246) circle (\grosseur);

\fill[racine] (2.00000000000000,1.62529276184047) circle (\grosseur);
\fill[racine] (1.82210585078189,1.30724968143271) circle (\grosseur);
\fill[racine] (2.17789414921811,1.30724968143271) circle (\grosseur);

\fill[racine] (2.23431457505076,1.43487787042860) circle (\grosseur);
\fill[racine] (2.15801714711854,1.54824552420293) circle (\grosseur);
\fill[racine] (1.84198285288146,1.54824552420293) circle (\grosseur);
\fill[racine] (1.76568542494924,1.43487787042860) circle (\grosseur);

\end{tikzpicture}
\caption{The join of $s_\alpha$ and $s_\gamma s_\beta$ exists and is the element $g=s_\gamma (s_\alpha s_\beta)^2$ whose inversion set is constituted of the roots in the blue triangle that represents $\cone(\gamma, {s_\gamma(\beta)},\alpha)$. However the join of $s_\alpha$ and $s_\beta s_\gamma$ does not exist since the red dot is contained in the  red triangle that represents $\cone(\beta, {s_\beta(\gamma)},\alpha)$. }
\label{fig:AffineC}
\end{figure}
\end{ex}

The end of the proof of Theorem~\ref{thm:Main1} gives the following useful characterization of inversion sets, see Figure~\ref{Fig:CaracFinite} for an illustration.

\begin{cor}\label{cor:CaracFinite} Let $H$ be a hyperplane that does not intersect the imaginary convex body $\conv(E)$. Then then  the positive roots on the other side of $\conv(E)$ relatively to $H$ form an inversion set $N(w)$ for a $w\in W$. Moreover, any finite biclosed set is obtained in this way.
\end{cor}

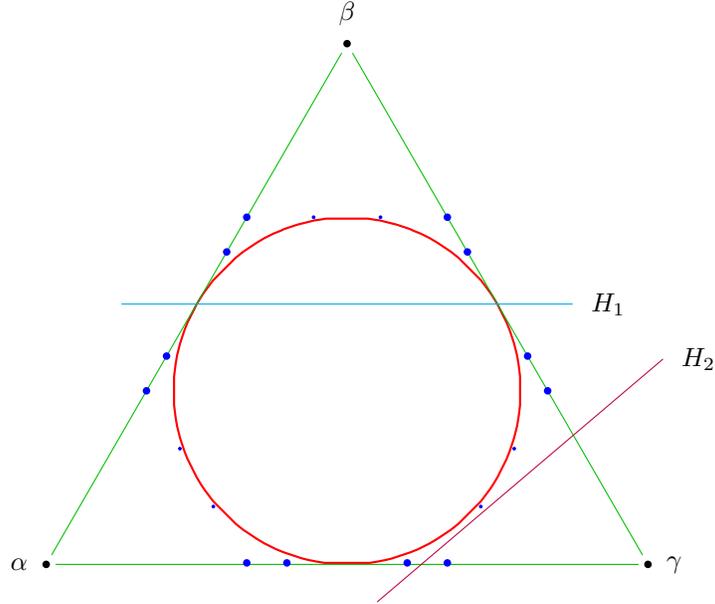
\begin{figure}[h!]
\begin{tikzpicture}
	[scale=2,
	 q/.style={red,line join=round,thick},
	 racine/.style={blue},
	 racinesimple/.style={black},
	 racinedih/.style={blue},
	 limit/.style={fill=red,draw=black,diamond},
	 limdir/.style={fill=orange,draw=black,diamond},
	 weight/.style={fill=green,draw=black,diamond},
	 sommet/.style={inner sep=2pt,circle,draw=black,fill=blue,thick,anchor=base},
	 rotate=0]

\def\grosseur{0.0125}
\def\grosseursimple{0.025}

\def\grosseurdih{0.0075}


\draw[q] (3.15,1.24) -- (3.15,1.25) -- (3.14,1.33) -- (3.13,1.39) -- (3.12,1.43) -- (3.11,1.47) -- (3.10,1.50) -- (3.09,1.53) -- (3.08,1.56) -- (3.06,1.61) -- (3.05,1.63) -- (3.01,1.71) -- (3.00,1.73) -- (2.97,1.78) -- (2.95,1.81) -- (2.92,1.85) -- (2.89,1.89) -- (2.88,1.90) -- (2.75,2.03) -- (2.74,2.04) -- (2.69,2.08) -- (2.66,2.10) -- (2.63,2.12) -- (2.60,2.14) -- (2.55,2.17) -- (2.53,2.18) -- (2.51,2.19) -- (2.49,2.20) -- (2.44,2.22) -- (2.42,2.23) -- (2.39,2.24) -- (2.36,2.25) -- (2.33,2.26) -- (2.25,2.28) -- (2.21,2.29) -- (2.20,2.29) -- (2.14,2.30) -- (2.13,2.30) -- (1.87,2.30) -- (1.86,2.30) -- (1.80,2.29) -- (1.79,2.29) -- (1.75,2.28) -- (1.67,2.26) -- (1.64,2.25) -- (1.61,2.24) -- (1.58,2.23) -- (1.56,2.22) -- (1.51,2.20) -- (1.49,2.19) -- (1.47,2.18) -- (1.45,2.17) -- (1.40,2.14) -- (1.37,2.12) -- (1.34,2.10) -- (1.31,2.08) -- (1.26,2.04) -- (1.25,2.03) -- (1.12,1.90) -- (1.11,1.89) -- (1.08,1.85) -- (1.05,1.81) -- (1.03,1.78) -- (1.00,1.73) -- (0.990,1.71) -- (0.950,1.63) -- (0.940,1.61) -- (0.920,1.56) -- (0.910,1.53) -- (0.900,1.50) -- (0.890,1.47) -- (0.880,1.43) -- (0.870,1.39) -- (0.860,1.33) -- (0.850,1.25) -- (0.850,1.24) -- (0.850,1.07) -- (0.850,1.06) -- (0.860,0.980) -- (0.870,0.920) -- (0.880,0.880) -- (0.890,0.840) -- (0.900,0.810) -- (0.910,0.780) -- (0.920,0.750) -- (0.940,0.700) -- (1.00,0.580) -- (1.03,0.530) -- (1.05,0.500) -- (1.08,0.460) -- (1.11,0.420) -- (1.12,0.410) -- (1.26,0.270) -- (1.31,0.230) -- (1.40,0.170) -- (1.45,0.140) -- (1.47,0.130) -- (1.49,0.120) -- (1.51,0.110) -- (1.53,0.100) -- (1.58,0.0800) -- (1.61,0.0700) -- (1.67,0.0500) -- (1.70,0.0400) -- (1.74,0.0300) -- (1.79,0.0200) -- (1.85,0.0100) -- (1.86,0.0100) -- (2.14,0.0100) -- (2.15,0.0100) -- (2.21,0.0200) -- (2.26,0.0300) -- (2.30,0.0400) -- (2.33,0.0500) -- (2.39,0.0700) -- (2.42,0.0800) -- (2.47,0.100) -- (2.49,0.110) -- (2.51,0.120) -- (2.53,0.130) -- (2.55,0.140) -- (2.60,0.170) -- (2.69,0.230) -- (2.74,0.270) -- (2.88,0.410) -- (2.89,0.420) -- (2.92,0.460) -- (2.95,0.500) -- (2.97,0.530) -- (3.00,0.580) -- (3.06,0.700) -- (3.08,0.750) -- (3.09,0.780) -- (3.10,0.810) -- (3.11,0.840) -- (3.12,0.880) -- (3.13,0.920) -- (3.14,0.980) -- (3.15,1.06) -- (3.15,1.07) -- cycle;

\node[label=left :{$\alpha$}] (a) at (0.000000000000000,0.000000000000000) {};
\fill[racinesimple] (0.000000000000000,0.000000000000000) circle (\grosseursimple);
\node[label=right :{$\gamma$}] (b) at (4.00000000000000,0.000000000000000) {};
\fill[racinesimple] (4.00000000000000,0.000000000000000) circle (\grosseursimple);
\node[label=above :{$\beta$}] (g) at (2.00000000000000,3.46410161513775) {};
\fill[racinesimple] (2.00000000000000,3.46410161513775) circle (\grosseursimple);

\draw[green!75!black] (a) -- (b) -- (g) -- (a);

\fill[racine] (1.33333333333333,2.30940107675850) circle (\grosseursimple);
\fill[racine] (1.33333333333333,0.01) circle (\grosseursimple);
\fill[racine] (0.666666666666667,1.15470053837925) circle (\grosseursimple);
\fill[racine] (2.66666666666667,2.30940107675850) circle (\grosseursimple);
\fill[racine] (3.33333333333333,1.15470053837925) circle (\grosseursimple);
\fill[racine] (2.66666666666667,0.01) circle (\grosseursimple);

\fill[racine] (0.800000000000000,1.38564064605510) circle (\grosseursimple);
\fill[racine] (1.60000000000000,0.01) circle (\grosseursimple);
\fill[racine] (1.20000000000000,2.07846096908265) circle (\grosseursimple);
\fill[racine] (2.80000000000000,2.07846096908265) circle (\grosseursimple);
\fill[racine] (0.888888888888889,0.769800358919501) circle (\grosseur);
\fill[racine] (2.88888888888889,0.384900179459750) circle (\grosseur);
\fill[racine] (2.22222222222222,2.30940107675850) circle (\grosseur);
\fill[racine] (2.40000000000000,0.01) circle (\grosseursimple);
\fill[racine] (3.11111111111111,0.769800358919501) circle (\grosseur);
\fill[racine] (1.77777777777778,2.30940107675850) circle (\grosseur);
\fill[racine] (1.11111111111111,0.384900179459750) circle (\grosseur);
\fill[racine] (3.20000000000000,1.38564064605510) circle (\grosseursimple);

\draw[cyan] ($(a)!0.5!(g)-(0.5,0)$) -- ($(b)!0.5!(g)+(0.5,0)$);

\node[label=right:{$H_1$}] at ($(b)!0.5!(g)+(0.5,0)$) {};

\draw[purple] ($(a)!0.5!(b)+(0.2,-0.25)$) -- ($(b)!0.25!(g)+(0.6,0.5)$);

\node[label=right:{$H_2$}] at ($(b)!0.25!(g)+(0.6,0.5)$) {};

\end{tikzpicture}
\caption{Illustration of Corollary~\ref{cor:CaracFinite}:  the set of roots on the side of $H_2$ not containing $\conv(E)$ forms a finite biclosed set, whereas neither side of $H_1$ provides a finite biclosed set.}
\label{Fig:CaracFinite}
\end{figure}
%

\subsection{Toward an even simpler criterion for the existence of a join}  As we stated after Theorem~\ref{thm:Dyer}, Dyer asks the question if we could replace the bounded hypothesis in Theorem~\ref{thm:Dyer} by ``$\overline{N(X)}$ is finite". We believe it to be true and state it as a conjecture:
\begin{qu}[{Dyer~\cite[Remarks 1.5]{dyer_weak_2011}}]\label{conj:2}
    Does $\bigvee X$ exists if $\overline{N(X)}$ is finite?
\end{qu}

In regard of Theorem~\ref{thm:Main1}, a strategy for answering Question~\ref{conj:2} would be to study  $\conv(\h N(X))\cap\conv(E)$ whenever $\overline{N(X)}$ is finite. Indeed, since  conical  closure implies $2$-closure, we have $\overline{N(X)}\subseteq \cone_\Phi(N(X))$; therefore $\cone(\overline{N(X)})\subseteq \cone(N(X))$ and by minimality of the conic closure we have 
$
\cone(\overline{N(X)})= \cone(N(X)).
$ 
 So if $\conv(\h N(X))\cap\conv(E)=\varnothing$, this would imply that $N(X)$ is bounded and so the join would exist. This discussion leads us naturally to state the following question.

\begin{qu}\label{conj:3}
    Does the join $\bigvee X$ exists if $\cone_\Phi(N(X))$ is finite? 
\end{qu}

Note that if $\bigvee X$ exists then $\cone_\Phi(N(X))$ is finite by Proposition~\ref{prop:JoinFinite}. Note also that if $\cone_\Phi(N(X))$ is finite, so is $N(X)$. Therefore $P=\conv(\h N(X))$ is a polytope (and so is compact).  One strategy to answer by the positive Question~\ref{conj:3} would be to prove that $P$ and $\conv(E)$ are strictly separated. The difficulty to answer Question~\ref{conj:3} lies then in the fact that the set of the limit root does not restrict well to root subsystems, see~\cite{hohlweg_asymptotical_2014,dyer_imaginary2_2013} for a discussion.  More precisely, if $A\subseteq \Phip$ is finite and if $\conv(E)$ intersects a proper face $F$ of the polytope $\conv(\h A)$, we would need to show that this face contains a limit root $x\in E(\Phi')$ where $\Phi'=\Phi\cap \textrm{aff}(F)$ (here $\textrm{aff}(F)$ is the affine subspace generated by $F$), which would imply that $\cone_\Phi(A)$ is infinite. This would require a continuation of the study of the faces of the imaginary convex body $\conv(E)$ initiated in~\cite{dyer_imaginary_2013} and \cite[\S3]{dyer_imaginary2_2013}. However, it is not too difficult to show that if $W$ is of rank $\leq 3$, then Question~\ref{conj:3} has a positive answer by following the above outlined strategy.

\section{Biclosed sets corresponding to inversion sets of infinite words}\label{se:Inf}

The aim of this section is to discuss a generalization of Corollary~\ref{cor:CaracFinite} to  inversion sets of infinite reduced  words on $S$. 

\subsection{Infinite reduced words and limit roots}\label{see:Inf} Let $\omega=s_1s_2s_3 \dots\in S^*$ be an infinite word on $S$. We say that $\omega$ is {\em an infinite reduced word} if any prefix $w_i=s_1\dots s_i$, $i\geq 1$ is reduced in $W$. We denote by $\mathcal W$ the set of infinite reduced words on $S$. Two infinite reduced words $w'$ and $w$ are equivalent, denoted by $\omega\sim \omega'$, if $\omega'$ can be obtained  from $\omega$ by a possibly infinite number of braid moves. More precisely,  define a preorder $\prec$ as follows: $\omega\prec \omega'$ if for any prefix $w_i$ of $\omega$, there is a prefix $w'_j$ of $\omega'$ such that $w_i\leq w'_j$. Then we write $\omega\sim \omega'$ if $\omega\prec \omega'$  and $\omega'\prec \omega$. 

Let $\mathcal W^\infty$ be the quotient set of $\mathcal W$ by this equivalence relation. Therefore, the partial order $\prec$ on $\mathcal W$ induces a partial order $\leq$ on $\mathcal W^\infty$ called the {\em limit weak order}; see~\cite[\S4.6]{lam_total_2013}. We say that {\em $w\in W$ is a prefix of $\omega\in \mathcal W^\infty$} if a reduced expression of $w$ is a prefix of some infinite reduced word $\omega'$ in the equivalence class of $\omega$, i.e., there is a prefix of $w'_j$ of $\omega'$ such that $w\leq w'_j$.

 For $\omega=s_1s_2\dots\in \mathcal W$, where $w_i=s_1\dots s_i$, we consider the following sequence of  roots:

\begin{equation}\label{eq:InfInv}
\beta_1:=\alpha_{s_1}\in\Delta\textrm{ and } \beta_i:=w_{i-1}(\alpha_{s_i})\in\Phip, \textrm{ for }i\geq 2.
\end{equation}
The {\em inversion set of $\omega$} is then
$$
N(\omega):=\bigcup_{i\in\mathbb N} N(w_i)=\{\beta_i\,|\, i\in \mathbb N^*\}\subseteq \Phip.
$$
We illustrate this notion in Figure~\ref{fig:InfiniteDi}, in~Example~\ref{ex:A2tilde} and~Figure~\ref{fig:A2tilde}.

\begin{ex}\label{ex:A2tilde} Consider the affine Coxeter group of type $\tilde A_2$ given by the following Coxeter graph:
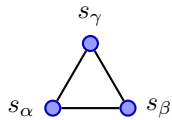
\begin{center}
\begin{tikzpicture}[sommet/.style={inner sep=2pt,circle,draw=blue!75!black,fill=blue!40,thick},]
\coordinate (ancre) at (0,0);
\node[sommet,label=left:$s_\alpha$] (alpha) at (ancre) {};
\node[sommet,label=right :$s_\beta$] (beta) at ($(ancre)+(1,0)$) {} edge[thick] (alpha);
\node[sommet,label=above:$s_\gamma$] (gamma) at ($(ancre)+(0.5,0.86)$) {} edge[thick] (alpha) edge[thick] (beta);
\end{tikzpicture}
\end{center}
Consider $\omega=(s_\alpha s_\beta s_\gamma)^\infty$ and $\omega'=s_\beta(s_\alpha s_\beta s_\gamma)^\infty=s_\beta\omega$: these two infinite reduced words are distinct in $\mathcal W^\infty$ since $s_\beta\prec \omega'$ and there is no prefix $w_i$ of $\omega$ such that $s_\beta\leq w_i$, but $\omega\prec\omega'$ since: 
$s_\alpha\leq s_\beta s_\alpha s_\beta$; $ s_\alpha s_\beta\leq s_\beta s_\alpha s_\beta$; $ s_\alpha s_\beta s_\gamma\leq s_\beta s_\alpha s_\beta s_\gamma s_\alpha=s_\alpha s_\beta s_\gamma s_\alpha s_\gamma$;  $s_\alpha s_\beta s_\gamma s_\alpha \leq s_\beta s_\alpha s_\beta s_\gamma s_\alpha=s_\alpha s_\beta s_\gamma s_\alpha s_\gamma;$ etc.; see Figure~\ref{fig:A2tilde}.

\begin{figure}[h!]
\begin{tikzpicture}
	[scale=2,
	 q/.style={red,line join=round,thick},
	 racine/.style={blue},
	 racinesimple/.style={inner sep=2pt,circle,draw=blue!75!black,fill=blue!40,thick},
	 racinedih/.style={blue},
	 rotate=0]

\def\grosseur{0.025}
\def\grosseursimple{0.05}

\def\grosseurdih{0.0075}

\filldraw[thick, red, fill opacity=0.5] (0.000000000000000,0.000000000000000) -- (2,0) -- (2.00000000000000,1.15470053837925) -- cycle;

\fill[red] (2.00000000000000,1.15470053837925) circle (\grosseursimple);

\node[racinesimple, label=left :{$\alpha$}] (a) at (0,0) {};
\node[racinesimple, label=right :{$\beta$}] (b) at (4,0) {};
\node[racinesimple, label=above:{$\gamma$}] (g) at (2.00000000000000,3.46410161513775) {};
\draw[green!75!black] (a) -- (b) -- (g) -- (a);
\fill[racine] (3.00000000000000,1.73205080756888) circle (\grosseur);
\fill[racine] (2.00000000000000,0.01) circle (\grosseur);
\fill[racine] (1.00000000000000,1.73205080756888) circle (\grosseur);

\fill[racine] (2.50000000000000,0.866025403784439) circle (\grosseur);
\fill[racine] (1.50000000000000,0.866025403784439) circle (\grosseur);
\fill[racine] (2.00000000000000,1.73205080756888) circle (\grosseur);

\fill[racine] (2.00000000000000,0.692820323027551) circle (\grosseur);
\fill[racine] (2.40000000000000,1.38564064605510) circle (\grosseur);
\fill[racine] (1.60000000000000,1.38564064605510) circle (\grosseur);

\fill[racine] (1.71428571428571,0.989743318610787) circle (\grosseur);
\fill[racine] (2.28571428571429,0.989743318610787) circle (\grosseur);
\fill[racine] (2.00000000000000,1.48461497791618) circle (\grosseur);

\fill[racine] (1.75000000000000,1.29903810567666) circle (\grosseur);
\fill[racine] (2.25000000000000,1.29903810567666) circle (\grosseur);
\fill[racine] (2.00000000000000,0.866025403784439) circle (\grosseur);

\fill[racine] (2.20000000000000,1.03923048454133) circle (\grosseur);
\fill[racine] (1.80000000000000,1.03923048454133) circle (\grosseur);
\fill[racine] (2.00000000000000,1.38564064605510) circle (\grosseur);

\fill[racine] (1.81818181818182,1.25967331459555) circle (\grosseur);
\fill[racine] (2.18181818181818,1.25967331459555) circle (\grosseur);
\fill[racine] (2.00000000000000,0.944754985946660) circle (\grosseur);

\fill[racine] (2.00000000000000,1.33234677505298) circle (\grosseur);
\fill[racine] (1.84615384615385,1.06587742004239) circle (\grosseur);
\fill[racine] (2.15384615384615,1.06587742004239) circle (\grosseur);

\fill[racine] (1.85714285714286,1.23717914826348) circle (\grosseur);
\fill[racine] (2.00000000000000,0.989743318610787) circle (\grosseur);
\fill[racine] (2.14285714285714,1.23717914826348) circle (\grosseur);


\node at ($(0,0)!0.5!(2,0)!0.33!(2,1.15470053837925)$) {A};

\node at (-0.5,1.5) {$\h N(\omega)=A\cap\h\Phi$};
\node at (4.5,1.5) {$\begin{array}{rcl}\h N(\omega')&=&\{\h\beta\}\cup \h{s_\beta( N(\omega))}\\&=& \{\h\beta\}\cup (A\cap\h\Phi)\end{array}$};

\end{tikzpicture}

\caption{The inversion sets of $\omega=(s_\alpha s_\beta s_\gamma)^\infty$ and $\omega'=s_\beta(s_\alpha s_\beta s_\gamma)^\infty$.}
\label{fig:A2tilde}
\end{figure}
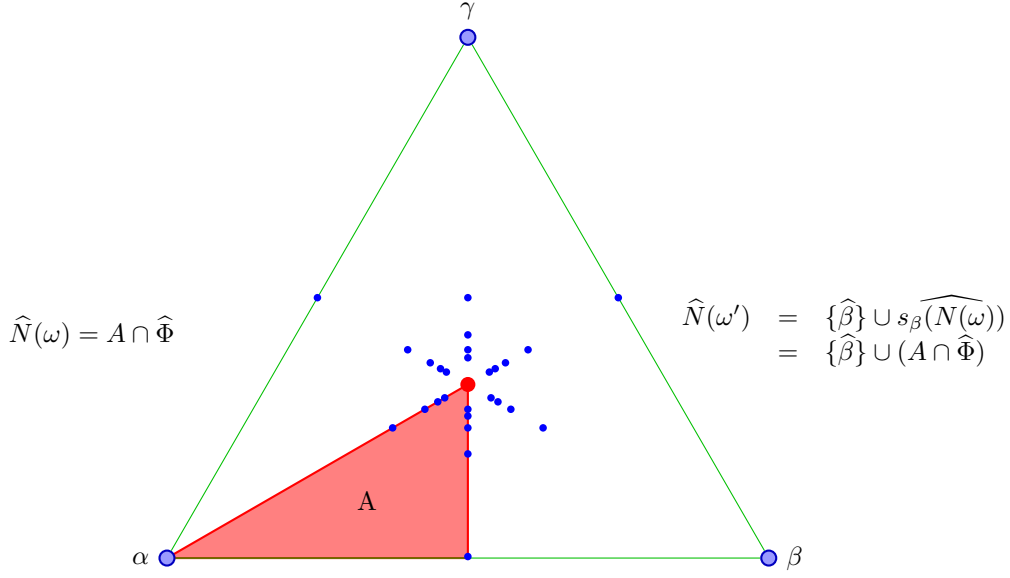
\end{ex}

\begin{prop}\label{prop:Infinite} Let $\omega,\omega'\in\mathcal W$.
\begin{enumerate}[(i)]
\item $N(\omega)$ is separable (hence biconvex and biclosed).
\item $w\in W$ is a prefix of $\omega$ if and only $N(w)\subseteq N(\omega)$.
\item If $\omega \prec \omega'$ if and only if $N(\omega)\subseteq N(\omega')$. 
\item $N(\omega)=N(\omega')$ if and and only if $\omega\sim \omega'$. 
\item $N:W\sqcup \mathcal W^\infty\rightarrow\mathcal B(\Phip)$ is a monomorphism.
\end{enumerate}
\end{prop}

\begin{remark}\begin{enumerate}[(a)]
\item The notion of inversion set of infinite words appear first, as far as we know,  in the work of Cellini and Papi~\cite{cellini_structure_1998} in the context of affine Weyl groups. The equivalence class, as well as the results stated in Proposition~\ref{prop:Infinite}, appears in the works of Ito~\cite{ito_classification_2001,ito_parameterizations_2005}. It was later used   in the work of Lam and Pylyavskyy~\cite{lam_total_2013}, and Lam and Thomas~\cite{lam_infinite_2013}. The proofs seem to have all been written within the framework of affine Coxeter systems and their associated crystallographic root systems; they generalize easily and, for convenience, we present them here in the context of an arbitrary Coxeter group. 
\item Lam and Pylyavskyy define in~\cite[\S4.6]{lam_total_2013} the {\em limit weak order in $\mathcal W^\infty$}: $\omega \leq \omega'$ in $\mathcal W^\infty$ if $N(\omega)\subseteq N(\omega')$. Thanks to Proposition~\ref{prop:Infinite} above, it is clear that the limit weak order is a subposet of the poset $(\mathcal B,\subseteq)$; recall that the poset $(\mathcal B,\subseteq)$ is conjectured to be a lattice (see Conjecture~\ref{conj:Dyer}). The limit weak order is studied further in Lam and Thomas~\cite{lam_infinite_2013}
\end{enumerate}
\end{remark}

\begin{proof} $(i)$ Assume that there is $\beta\in\cone(N(\omega))\cap\cone(N(\omega)^c)$. So there is a reduced prefix $w_i$ of $\omega$ such that $\beta \in \cone(N(w_i))$. Since $N(w_i) \subseteq N(\omega)$, we have $\beta \in \cone(N(\omega)^c)\subseteq \cone(N(w_i)^c)$. So $\beta\in\cone(N(w_i))\cap \cone(N(w_i)^c)$, implying $\beta=0$ since $N(w_i)$ is separable by Proposition~\ref{prop:FiniteBiclos}, a contradiction. 
	
\noindent $(ii)$ The direct implication follows from definition. Now, let $w\in W$ such that $N(w)\subseteq N(\omega)$. Since $N(w)$ is finite and $N(\omega)$ is the union of increasing finite subsets of $N(\omega)$, there is a prefix $w_i$ of $\omega$ such that $N(w)\subseteq N(w_i)$. Hence $w$ is a prefix of $w_i$ by Proposition~\ref{prop:biclosed}. Therefore $w$ is a prefix of $\omega$. 

\noindent $(iii)$ If $\omega\leq \omega'$, then any prefix $w\in W$ of $\omega$ is a prefix of $\omega'$, and therefore $N(w)\subseteq N(\omega)$ by $(ii)$. Hence $N(\omega)\subseteq N(\omega')$. Now, assume  that $N(\omega)=N(\omega')$. Let $w_i$ a finite reduced prefix of $\omega$, so $N(w_i)\subseteq N(\omega)=N(\omega')$. Therefore $w_i$ is a prefix of $\omega'$ by $(ii)$. Therefore $\omega\prec \omega'$. We show similarly that $\omega'\prec\omega$, hence that $\omega\sim \omega'$. Finally $(iv)$ follows from $(iii)$, and $(v)$ follows from $(iv)$.
\end{proof}

We state now a corollary of Theorem~\ref{thm:Main1}, extending partially Corollary~\ref{cor:CaracFinite} to infinite reduced words.

\begin{cor}\label{prop:ImN} Let $\omega\in \mathcal W^\infty$ then $\conv(\h N(\omega))\cap \conv(E)=\varnothing$. 
\end{cor}
\begin{proof} If there is $x \in \conv(\h N(\omega))\cap \conv(E)$, then there are roots~${{\beta_1,\dots,\beta_k\in N(\omega)}}$ such that 
$
x=\sum_{i=1}^k a_i\beta_i,\ a_i\geq 0.
$
By definition of $N(\omega)$, there is a prefix~${{w_i\in W}}$ of $\omega$ such that $\beta_1,\dots,\beta_k\in N(w_i)$. Hence $x\in\cone(N(w_i))$. Therefore $x\in \conv(\h N(w_i))\cap \conv(E)$ contradicting  Theorem~\ref{thm:Main1}.
\end{proof}

\begin{remark}[Separable does not mean separable by an hyperplane] It is not true in general that an infinite subset of $\Phip$ is separable if and only if there is an hyperplane that separates it from its complement. Take the universal Coxeter group of rank $3$ generated by $s,t,r$ as in Figure~\ref{fig:Univ}. Then $A=N(s(t r)^\infty)$ is an infinite separable set by Proprosition~\ref{prop:Infinite} and $x=\h{s(\alpha_r+\alpha_t)}\in E$ is the accumulation point of  $A$. But $A$ is not strictly separated by an hyperplane from its complement. Indeed, assume by contradiction that there is an hyperplane~$H$ separating $\h A$ and $\h A^c$. Since in this case $E=\h Q$ (see \cite[Theorem 4.4]{dyer_imaginary2_2013}), the only point of $\h Q$ contained in the hyperplane $H$ separating $\h A$ and $\h A^c$ has to be $\h x$. Therefore,  $H$ is tangent to the circle and contains   $N(s(rt)^\infty)\subseteq A^c$, a contradiction. Such a counterexample that lies in the affine Coxeter group of type $\tilde A_2$ can be found in the second author's thesis~\cite[Figure~2.11]{labbe_polyhedral_2013}.
\end{remark}

\subsection{Toward a geometric characterization of the inversion sets of infinite reduced words}

Corollary~\ref{cor:CaracFinite} states that a biclosed set $A$ is an inversion set for some  $w\in W$ if and only if there is a hyperplane $H$ not intersecting $\conv(E)$ such that $A$ lies on the side of $H$ and its complement lies with $\conv(E)$ on the other side. By Proposition~\ref{prop:FiniteBiclos}, this  is equivalent to state that $\acc(\h A)$ is a empty (i.e. $A$ finite),  $A$ is biconvex   and $\conv(\h A)\cap \conv(E)=\varnothing$. 

We discuss now a conjectural generalization of this characterization for the case of inversion sets of infinite words. First, say that an infinite reduced word $\omega$ is {\em connected} if the induced subgraph of the Coxeter graph corresponding to the letters that appear an infinite number of times in $\omega$ is connected. 

\begin{conject}\label{conj:CharInf}  Let $A\in\mathcal B$. Then $A=N(\omega)$ for some connected  $\omega\in \mathcal W^\infty$ if and only if $A$ is biconvex,  $\acc(\h A)$ is a singleton and $\conv(\h A)\cap \conv(E)=\varnothing$.
\end{conject}

\begin{remark} \label{rem:CePa}
\begin{enumerate}[(a)]
\item  Assume that $A=N(\omega)$, for some $\omega\in \mathcal W^\infty$. Then  $A$ is biconvex  and $\conv(\h A)\cap \conv(E)=\varnothing$, by Proposition~\ref{prop:Infinite} and Corollary~\ref{prop:ImN}.

\item By a theorem of Cellini and Papi~\cite[Theorem~3.12]{cellini_structure_1998}, this conjecture is true in affine types, as mentioned in~\cite[Proposition~2.6]{baumann_affine_2014}. Indeed in the affine type, $E=\{\h\delta\}$ is a singleton, and $\delta$ is called the {\em imaginary root}. So if $\cone(A)$ does not contain $\delta$,  we apply Cellini and Papi's result and obtain an infinite reduced word $\omega$ such that $A=N(\omega)$.  Moreover the accumulation points of an infinite $\h A$ is in $E$, so is a singleton.

\item This conjecture is true for infinite dihedral groups. 

\item As was pointed out by one of the anonymous referee, the definition of connected infinite reduced word is needed in order for the conjecture above, and the two conjectures below, to hold. Indeed, let for instance $(W,S)$ to be the infinite Coxeter system with Coxeter graph
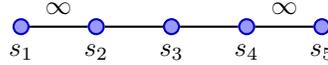
\begin{center}
\begin{tikzpicture}[sommet/.style={inner sep=2pt,circle,draw=blue!75!black,fill=blue!40,thick},] 
      \coordinate (ancre) at (0,-1.5);
      \node[sommet,label=below:$s_1$] (t1) at (ancre) {};
      \node[sommet,label=below :$s_2$] (t2) at ($(ancre)+(1,0)$) {} edge[thick] node[auto,swap] {} (t1);
      \node[sommet,label=below :$s_3$] (t3) at ($(ancre)+(2,0)$) {} edge[thick] node[auto,swap] {} (t2);
    	 \node[sommet,label=below :$s_4$] (t4) at ($(ancre)+(3,0)$) {} edge[thick] node[auto,swap] {} (t3);
      \node[sommet,label=below :$s_5$] at ($(ancre)+(4,0)$) {} edge[thick] node[auto,swap] {} (t4);
            \node[black] at (0.5,-1.25) {$\infty$};
            \node[black] at (3.5,-1.25) {$\infty$};
\end{tikzpicture} 
\end{center}
Then the inversion set of the infinite and not connected reduced word $(s_1s_2s_4s_5)^\infty$ has two accumulation points: the one corresponding to $N((s_1s_2)^\infty)$ and the one corresponding to $N((s_4s_5)^\infty)$. 

\end{enumerate}
\end{remark}

The left-to-right implication in Conjecture~\ref{conj:CharInf} follows from the following weaker conjecture.  

\begin{conject}\label{conj:LimN}  Let $w$ be a  connected  infinite reduced word. Then Acc$(\h N(w))$ is a singleton.
\end{conject}

\begin{remark} \label{rem:Lor}
\begin{enumerate}[(a)]

\item This conjecture is equivalent to showing that the injective sequence $(\h\beta_n)_{n\in\mathbb N^*}$ converges, where $\h\beta_n$ is defined in Equation~(\ref{eq:InfInv}).  Indeed, the sequence $(\h\beta_n)_{n\in\mathbb N^*}$ is bounded, and therefore the equivalence follows from the following general topological fact based on the Bolzano--Weierstra{\ss} Theorem: a bounded sequence in a  compact  metric space converges if and only if it has a unique accumulation point.

\item Conjecture~\ref{conj:LimN} is obviously true for irreducible affine types, since $E$ is a singleton.

\item  Conjecture~\ref{conj:LimN}, as well as the left-to-right statement in Conjecture~\ref{conj:CharInf}, are true in is true for based root systems of Lorentzian type, i.e., the signature of the root system is $(n-1,1)$, as shown by Chen and the second author in~\cite[Theorem~2.8]{chen_limit_2014}. Note that in the proof, the authors use the fact that $\h Q$ is strictly convex in this case and their approach works also for linearly dependent bases, see~\cite[Remark~2.4]{chen_limit_2014}. Maybe the argument of strict convexity of $Q$  could be replaced in full generality by using \cite[Corollary 6.8 (ii)]{dyer_imaginary2_2013} and a result similar to \cite[Proposition 2.1]{chen_limit_2014}. In the Lorentzian case, this gives an alternate - but less straightforward - proof of \cite[Theorem 2.8]{chen_limit_2014}.  

\item Conjecture~\ref{conj:LimN} is true for rank $2$ and $3$ infinite root systems, since they are all affine or Lorentzian in this case. So  the left-to-right statement in Conjecture~\ref{conj:CharInf} is true in those cases. We could not find any counterexamples to the right-to-left statement of Conjecture~\ref{conj:CharInf}.

\end{enumerate}
\end{remark}

\subsection{Limit sets of inversion sets and paths in the imaginary cone}

It would be interesting to better understand  biclosed sets in relation with their limit roots and subsets of the imaginary cone. Here, we propose an approach to solve Conjecture~\ref{conj:LimN}. 

The imaginary convex body $\conv(E)$ has a {\em tiling} parameterized by $W$: let $K=\{u\in\conv(\h \Delta)\,|\, B(u,\alpha_s)\leq 0, \ \forall s\in S\}$, then $\conv(E)=\overline{W\cdot K}$, where $w\cdot x=\widehat{w(x)}$ is well-defined  on $\conv(E)$ for any $w\in W$, see~\cite{dyer_imaginary_2013,dyer_imaginary2_2013}. 

Assume from now on that $(\Phi,\Delta)$ is an irreducible indefinite based root system (i.e. not finite nor affine).  So the relative interior of $K$ is nonempty and open in the linear span of $K$.  Take $z$ in the relative interior of $K$, i.e., $B(z,\alpha)< 0$ for all $\alpha\in\Delta$.  Then $\omega \in \mathcal W$  can be seen, with the same notations as in \S\ref{see:Inf}, as a sequence   $z$, $w_1\cdot z$, $w_2\cdot z$, $\dots$, $w_n\cdot z$ within $\conv(E)$. Note that $w_n\cdot z$ is an injective sequence, since $w_n\cdot z\in w\cdot K$ and $K$ is a fundamental domain for the action of $W$ on $W\cdot K$ (see Figure~\ref{fig:Imaginary}).

\begin{conject}\label{conj:InfiniteWord} Let $\omega$ be a  connected  infinite reduced word and $(w_n)_{n\in\mathbb N^*}$ the increasing sequence of its prefixes. Let $z$ in the relative interior of $K$. Then there is $x\in E$ such that 
$$
\acc((w_n\cdot z)_{n\in\mathbb N})=\acc(\h N(w))=\{x\}.
$$ 
In other words both sequences $(w_n\cdot z)_{n\in\mathbb N^*}$ and $(\h \beta_n)_{n\in\mathbb N^*}$ converge to $x\in E$.
\end{conject}

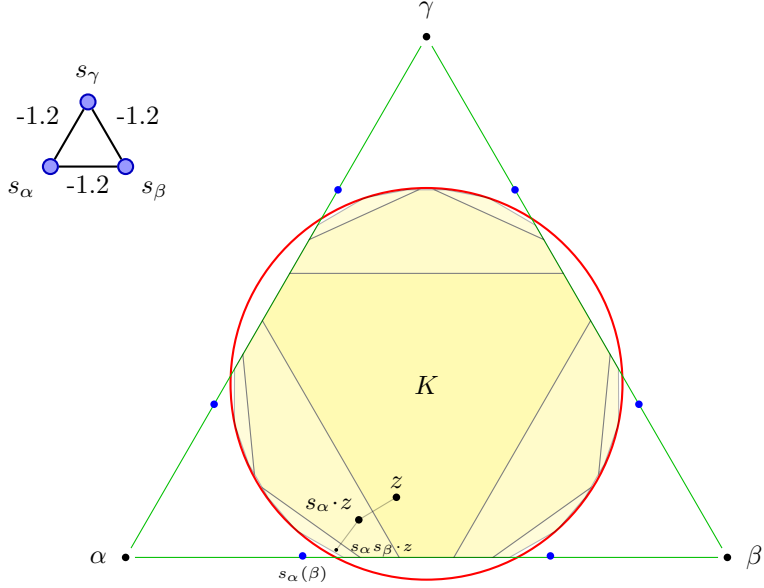
\begin{figure}
\begin{tikzpicture}
	[scale=2,
	 q/.style={red,thick},
	 racine/.style={blue},
	 racinesimple/.style={black},
	 racinedih/.style={blue},
	 sommet/.style={inner sep=2pt,circle,draw=blue!75!black,fill=blue!40,thick},
	 rotate=0]

\def\grosseur{0.0125}
\def\grosseursimple{0.025}

\def\grosseurdih{0.0075}

\draw[q] (2, 1.15470053837925) circle (1.30267789455786) ;

\node[label=left :{$\alpha$}] (a) at (0.000000000000000,0.000000000000000) {};
\fill[racinesimple] (0.000000000000000,0.000000000000000) circle (\grosseursimple);\node[label=right :{$\beta$}] (b) at (4.00000000000000,0.000000000000000) {};
\fill[racinesimple] (4.00000000000000,0.000000000000000) circle (\grosseursimple);\node[label=above :{$\gamma$}] (g) at (2.00000000000000,3.46410161513775) {};
\fill[racinesimple] (2.00000000000000,3.46410161513775) circle (\grosseursimple);
\draw[green!75!black] (a) -- (b) -- (g) -- (a);
\fill[racine] (1.17647058823529,0.0100000000000000) circle (\grosseursimple);
\fill[racine] (1.41176470588235,2.44524819892077) circle (\grosseursimple);
\fill[racine] (2.82352941176471,0.0100000000000000) circle (\grosseursimple);
\fill[racine] (3.41176470588235,1.01885341621699) circle (\grosseursimple);
\fill[racine] (0.588235294117647,1.01885341621699) circle (\grosseursimple);
\fill[racine] (2.58823529411765,2.44524819892077) circle (\grosseursimple);

\draw  (1.17647058823529,-0.1100000000000000) node{{\tiny $s_\alpha(\beta)$}};

\coordinate (ancre) at (-0.5,2.6);
\node[sommet,label=below left:$s_\alpha$] (alpha) at (ancre) {};
\node[sommet,label=below right :$s_\beta$] (beta) at ($(ancre)+(0.5,0)$) {} edge[thick] node[auto] {-1.2} (alpha);
\node[sommet,label=above:$s_\gamma$] (gamma) at ($(ancre)+(0.25,0.43)$) {} edge[thick] node[auto,swap] {-1.2} (alpha) edge[thick] node[auto] {-1.2} (beta);

\filldraw[draw= black ,fill= yellow ,opacity= 0.300000000000000 ]
(2.909090908, 1.8895099722) --
(1.0909090908, 1.8895099722) --
(0.909090909, 1.5745916434) --
(1.8181818182, 0.0) --
(2.181818182, 0.0) --
(3.09090909, 1.5745916426) --
cycle ;


\filldraw[draw= black ,fill= yellow!80 ,opacity= 0.300000000000000 ]
(0.855614973194145, 0.555738227281955)  -- 
(0.779220779101243, 1.34964998007078)  -- 
(0.909090909001093, 1.57459164340189)  -- 
(1.81818181816364, 0.000000000000000)  -- 
(1.55844155834879, 0.000000000000000)  -- 
(0.909090909200340, 0.463115189191957)  -- 
cycle ;

\filldraw[draw= black ,fill= yellow!80 ,opacity= 0.300000000000000 ]
(3.22077922077631, 1.34964997879825)  -- 
(3.14438502678509, 0.555738227066197)  -- 
(3.09090909093551, 0.463115189208216)  -- 
(2.44155844154916, 0.000000000000000)  -- 
(2.18181818163636, 0.000000000000000)  -- 
(3.09090909053703, 1.57459164166984)  -- 
cycle ;

\filldraw[draw= black ,fill= yellow!80 ,opacity= 0.300000000000000 ]
(2.90909090835412, 1.88950997158664)  -- 
(1.09090909044588, 1.88950997158664)  -- 
(1.22077922060431, 2.11445163513458)  -- 
(1.94652406417647, 2.44524819892077)  -- 
(2.05347593588235, 2.44524819892077)  -- 
(2.77922077811584, 2.11445163554275)  -- 
cycle ;


\filldraw[draw= black ,fill= yellow!60 ,opacity= 0.300000000000000 ]
(2.88170407005843, 0.197647486673624)  -- 
(3.05271199391595, 0.396955876465156)  -- 
(3.09090909093583, 0.463115189208773)  -- 
(2.44155844156772, 0.000000000000000)  -- 
(2.55161787368442, 0.000000000000000)  -- 
(2.85271933305869, 0.171898651017698)  -- 
cycle ;

\filldraw[draw= black ,fill= yellow!60 ,opacity= 0.300000000000000 ]
(1.57522826510843, 2.38457808676480)  -- 
(1.27580893667477, 2.20976589928114)  -- 
(1.22077922060375, 2.11445163513362)  -- 
(1.94652406416578, 2.44524819892077)  -- 
(1.87012987010259, 2.44524819892077)  -- 
(1.61201977942229, 2.39680518749896)  -- 
cycle ;

\filldraw[draw= black ,fill= yellow!60 ,opacity= 0.300000000000000 ]
(0.947288006310731, 0.396955876200963)  -- 
(1.11829592999777, 0.197647486608127)  -- 
(1.14728066692370, 0.171898651013539)  -- 
(1.44838212635149, 0.000000000000000)  -- 
(1.55844155853432, 0.000000000000000)  -- 
(0.909090909358289, 0.463115188918381)  -- 
cycle ;

\filldraw[draw= black ,fill= yellow!60 ,opacity= 0.300000000000000 ]
(2.72419106255821, 2.20976589972896)  -- 
(2.42477173484066, 2.38457808679453)  -- 
(2.38798022062887, 2.39680518749691)  -- 
(2.12987012986740, 2.44524819892077)  -- 
(2.05347593577540, 2.44524819892077)  -- 
(2.77922077784185, 2.11445163601732)  -- 
cycle ;

\filldraw[draw= black ,fill= yellow!60 ,opacity= 0.300000000000000 ]
(0.855614973298298, 0.555738227101555)  -- 
(0.779220779281918, 1.34964998038372)  -- 
(0.724191063174665, 1.25433571593848)  -- 
(0.722508932116326, 0.907624877326688)  -- 
(0.730315709398947, 0.869648941020573)  -- 
(0.817417876158433, 0.621897540049636)  -- 
cycle ;

\filldraw[draw= black ,fill= yellow!60 ,opacity= 0.300000000000000 ]
(3.22077922059563, 1.34964997911119)  -- 
(3.14438502668094, 0.555738226885798)  -- 
(3.18258212374129, 0.621897539719493)  -- 
(3.26968429060517, 0.869648941040586)  -- 
(3.27749106788816, 0.907624877348487)  -- 
(3.27580893678228, 1.25433571494016)  -- 
cycle ;

\draw  (2, 1.15) node{$K$};

\fill[racinesimple] (1.8, 0.4) circle (\grosseursimple);
\draw  (1.8, 0.5) node{$z$};

\draw  (1.35, 0.35) node{{\small $s_{\alpha}\!\cdot\!z$}};
\fill[racinesimple] (1.55, 0.25) circle (\grosseursimple);

\draw  (1.7, 0.05) node{{\tiny{$s_{\alpha}s_\beta\!\cdot\!z$}}};
\fill[racinesimple] (1.4, 0.05) circle (\grosseur);

\filldraw[draw= black ,fill= black!60 ,opacity= 0.300000000000000 ]
(1.8, 0.4)--(1.55, 0.25);
\filldraw[draw= black ,fill= black!60 ,opacity= 0.300000000000000 ]

(1.55, 0.25)--(1.4, 0.05);

\end{tikzpicture}
\caption{Illustration of Conjecture~\ref{conj:InfiniteWord}: the inversion set of $N((s_\alpha s_\beta)^\infty)$ has an  accumulation point, which is the limit of the sequence $z,s_\alpha\cdot z,s_\alpha s_\beta\cdot z,\dots$. This is the intersection point of $Q$ with $[\h\alpha,\h\beta]$ that is the closest to $\h\alpha$.  \label{fig:Imaginary}}
\end{figure}

\begin{remark}
\begin{enumerate}[(a)]
\item Conjecture~\ref{conj:InfiniteWord} has a taste of \cite[Theorem 3.11]{dyer_imaginary2_2013}. 

\item We may ask  a weaker question: do we have  
$
\acc((w_n\cdot z)_{n\in\mathbb N})= \acc(\h N(w))?
$ 

\item As stated in  \cite[Remark 6.16(b)]{dyer_imaginary2_2013}, we do not know if Acc$(W\cdot x)=E$ for $x\in W\cdot K$ in general. We know that Acc$(W\cdot x)\supseteq E$, see \cite[Corollary 6.15]{dyer_imaginary2_2013}. So Conjecture~\ref{conj:InfiniteWord} slightly improves  the statement of \cite[Corollary 6.15]{dyer_imaginary2_2013}: any sequence of the form $\{w_n\cdot x\}_{n\in\mathbb N}\subseteq W\cdot x$ where $w_n$ are increasing prefixes of an infinite reduced word on $S$ has an accumulation point in $E$.

\item It would be very interesting to fit Lam and Thomas' results~\cite{lam_infinite_2013} in our framework by interpreting the imaginary convex body as a realization of the Davis complex, see~\cite{davis_geometry_2008,abramenko_buildings_2008}.  It would allow the use of tools from CAT$(0)$ spaces to explore Conjecture~\ref{conj:InfiniteWord}.  Assume $(W,S)$ be an irreducible indefinite Coxeter system, otherwise the question below has  trivially a negative answer. In the setting of the Davis complex, the set $S$ is  called the {\em nerve of $(W,S)$}. Let $S'$ be the set of {\em spherical subsets of $S$}: $I\in S'$ if $W_I= \langle I\rangle$ is finite. A way to start is to take a point $z$ in the relative interior of $K$. For any $I \subset S'$, define 
 \[
 	P_I^z:=\conv(W_I\cdot z).
 \]
Then $P_I^z$ is the permutahedron of the finite Coxeter group $W_I$ with based point~$x$. Now take 
\[
	\mathcal D=W\cdot\left(\bigcup_{I\subseteq S'} P^z_I\right)\subseteq \conv(E).
\]
 Finally, assume that~$\mathcal D$ is endowed with the piecewise metric $\mu$ given by the Euclidean metric on each permutahedron $w\cdot P^z_I$, for any $w\in W$. Is $(\mathcal D,\mu)$ a geometric realization of the Davis complex? An approach to answer this problem can be found in~\cite[Appendix B4-B5]{krammer_conjugacy_2009}.

\end{enumerate}
\end{remark}

\section{Relationships between biclosed and biconvex sets}\label{se:OnDef}

We know that for a finite set, the notion of separable, biconvex and biclosed are the same, see Proposition~\ref{prop:FiniteBiclos}. In this final section, we discuss the differences between the definition of  biclosed, biconvex and separable in the case of infinite sets. 

\subsection{On infinite biclosed sets and biconvex sets}\label{sse:BiclosBiconv} We establish in the following examples that biclosed sets are not biconvex in general.  

\begin{ex}[A biclosed set that is not biconvex in a geometric representation of dimension $3$ of a Coxeter group of rank $4$] 
We consider  $(\mathcal U_3,S)$ to be the universal Coxeter system of rank $3$ whose Coxeter graph is: 
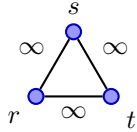
\begin{center}
\begin{tikzpicture}[sommet/.style={inner sep=2pt,circle,draw=blue!75!black,fill=blue!40,thick},]
\coordinate (ancre) at (0,0);
\node[sommet,label=below left:$r$] (alpha) at (ancre) {};
\node[sommet,label=below right :$t$] (beta) at ($(ancre)+(1,0)$) {} edge[thick] node[auto] {$\infty$} (alpha);
\node[sommet,label=above:$s$] (gamma) at ($(ancre)+(0.5,0.86)$) {} edge[thick] node[auto,swap] {$\infty$} (alpha) edge[thick] node[auto] {$\infty$} (beta);
\end{tikzpicture}

\end{center}
Let $W'$ the reflection subgroup of $\mathcal U_3$ generated by $S'=\{srs, sts, r,t\}$.  It is not difficult to see that $\Delta'=\{s(\alpha_r), s(\alpha_t), \alpha_r,\alpha_t\}$ is a simple system, with
$$
B( \alpha_r,\alpha_t)=B(s(\alpha_r), s(\alpha_t))=B(\alpha_r,s(\alpha_r))=B(\alpha_t, s(\alpha_t))=-1
$$
and
$$
B(\alpha_r,s(\alpha_t))=B(s(\alpha_r), \alpha_t)=-3.
$$ 
Note that $\Delta'$ is positively linearly independent but not linearly independent. It follows that $(W',S')$ is a universal Coxeter system of rank~$4$ (geometrically represented in dimension $3$). Now consider the standard parabolic subgroup $W_I$ of $(W',S')$ given by $I=\{sts, r,t\}\subseteq S'$, with associated simple system $\Delta_I=\{\alpha_t,\alpha_r, s(\alpha_t)\}$. The associated positive root subsystem $\PhipI=W_I(\Delta_I) \cap \Phi^+$ is a {\em biclosed set in $\Phipprime=W'(\Delta')\cap \Phip$}. But $\PhipI$ {\em is not biconvex in $\Phipprime$}. Indeed $s(\alpha_r)\notin \PhipI$ so $ts(\alpha_r)\notin \PhipI$. But 
$$
ts(\alpha_r)=\alpha_r+2\alpha_s+ 6\alpha_t =  \alpha_r + 5\alpha_t + s(\alpha_t)\in\cone_{\Phi'}(\PhipI).
$$
So $\cone_{\Phi'}(\PhipI)\not = \PhipI$. An illustration is given in Figure~\ref{fig:Univ}.
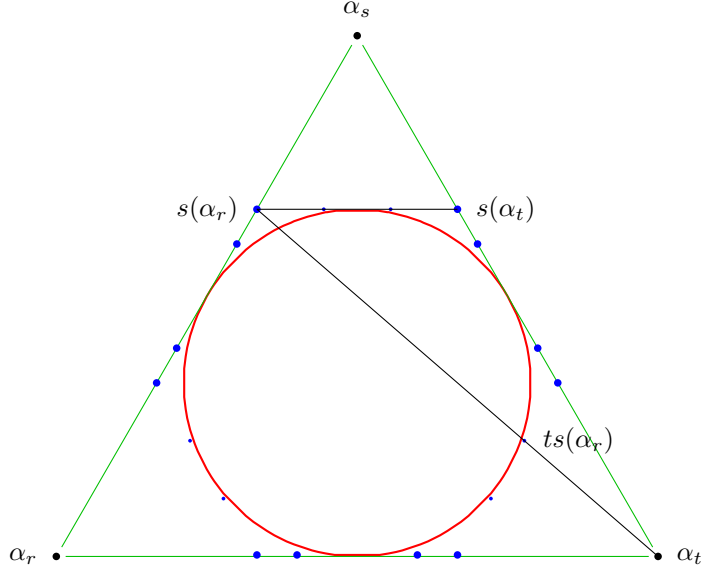
\begin{figure}[!htbp]
	\begin{center}
\begin{tikzpicture}
	[scale=2,
	 q/.style={red,line join=round,thick},
	 racine/.style={blue},
	 racinesimple/.style={black},
	 racinedih/.style={blue},
	 limit/.style={fill=red,draw=black,diamond},
	 limdir/.style={fill=orange,draw=black,diamond},
	 weight/.style={fill=green,draw=black,diamond},
	 sommet/.style={inner sep=2pt,circle,draw=black,fill=blue,thick,anchor=base},
	 rotate=0]

\def\grosseur{0.0125}
\def\grosseursimple{0.025}

\def\grosseurdih{0.0075}


\draw[q] (3.15,1.24) -- (3.15,1.25) -- (3.14,1.33) -- (3.13,1.39) -- (3.12,1.43) -- (3.11,1.47) -- (3.10,1.50) -- (3.09,1.53) -- (3.08,1.56) -- (3.06,1.61) -- (3.05,1.63) -- (3.01,1.71) -- (3.00,1.73) -- (2.97,1.78) -- (2.95,1.81) -- (2.92,1.85) -- (2.89,1.89) -- (2.88,1.90) -- (2.75,2.03) -- (2.74,2.04) -- (2.69,2.08) -- (2.66,2.10) -- (2.63,2.12) -- (2.60,2.14) -- (2.55,2.17) -- (2.53,2.18) -- (2.51,2.19) -- (2.49,2.20) -- (2.44,2.22) -- (2.42,2.23) -- (2.39,2.24) -- (2.36,2.25) -- (2.33,2.26) -- (2.25,2.28) -- (2.21,2.29) -- (2.20,2.29) -- (2.14,2.30) -- (2.13,2.30) -- (1.87,2.30) -- (1.86,2.30) -- (1.80,2.29) -- (1.79,2.29) -- (1.75,2.28) -- (1.67,2.26) -- (1.64,2.25) -- (1.61,2.24) -- (1.58,2.23) -- (1.56,2.22) -- (1.51,2.20) -- (1.49,2.19) -- (1.47,2.18) -- (1.45,2.17) -- (1.40,2.14) -- (1.37,2.12) -- (1.34,2.10) -- (1.31,2.08) -- (1.26,2.04) -- (1.25,2.03) -- (1.12,1.90) -- (1.11,1.89) -- (1.08,1.85) -- (1.05,1.81) -- (1.03,1.78) -- (1.00,1.73) -- (0.990,1.71) -- (0.950,1.63) -- (0.940,1.61) -- (0.920,1.56) -- (0.910,1.53) -- (0.900,1.50) -- (0.890,1.47) -- (0.880,1.43) -- (0.870,1.39) -- (0.860,1.33) -- (0.850,1.25) -- (0.850,1.24) -- (0.850,1.07) -- (0.850,1.06) -- (0.860,0.980) -- (0.870,0.920) -- (0.880,0.880) -- (0.890,0.840) -- (0.900,0.810) -- (0.910,0.780) -- (0.920,0.750) -- (0.940,0.700) -- (1.00,0.580) -- (1.03,0.530) -- (1.05,0.500) -- (1.08,0.460) -- (1.11,0.420) -- (1.12,0.410) -- (1.26,0.270) -- (1.31,0.230) -- (1.40,0.170) -- (1.45,0.140) -- (1.47,0.130) -- (1.49,0.120) -- (1.51,0.110) -- (1.53,0.100) -- (1.58,0.0800) -- (1.61,0.0700) -- (1.67,0.0500) -- (1.70,0.0400) -- (1.74,0.0300) -- (1.79,0.0200) -- (1.85,0.0100) -- (1.86,0.0100) -- (2.14,0.0100) -- (2.15,0.0100) -- (2.21,0.0200) -- (2.26,0.0300) -- (2.30,0.0400) -- (2.33,0.0500) -- (2.39,0.0700) -- (2.42,0.0800) -- (2.47,0.100) -- (2.49,0.110) -- (2.51,0.120) -- (2.53,0.130) -- (2.55,0.140) -- (2.60,0.170) -- (2.69,0.230) -- (2.74,0.270) -- (2.88,0.410) -- (2.89,0.420) -- (2.92,0.460) -- (2.95,0.500) -- (2.97,0.530) -- (3.00,0.580) -- (3.06,0.700) -- (3.08,0.750) -- (3.09,0.780) -- (3.10,0.810) -- (3.11,0.840) -- (3.12,0.880) -- (3.13,0.920) -- (3.14,0.980) -- (3.15,1.06) -- (3.15,1.07) -- cycle;

\node[label=left :{$\alpha_r$}] (a) at (0.000000000000000,0.000000000000000) {};
\fill[racinesimple] (0.000000000000000,0.000000000000000) circle (\grosseursimple);
\node[label=right :{$\alpha_t$}] (b) at (4.00000000000000,0.000000000000000) {};
\fill[racinesimple] (4.00000000000000,0.000000000000000) circle (\grosseursimple);
\node[label=above :{$\alpha_s$}] (g) at (2.00000000000000,3.46410161513775) {};
\fill[racinesimple] (2.00000000000000,3.46410161513775) circle (\grosseursimple);

\draw[green!75!black] (a) -- (b) -- (g) -- (a);

\fill[racine] (1.33333333333333,2.30940107675850) circle (\grosseursimple);
\fill[racine] (1.33333333333333,0.01) circle (\grosseursimple);
\fill[racine] (0.666666666666667,1.15470053837925) circle (\grosseursimple);
\fill[racine] (2.66666666666667,2.30940107675850) circle (\grosseursimple);
\fill[racine] (3.33333333333333,1.15470053837925) circle (\grosseursimple);
\fill[racine] (2.66666666666667,0.01) circle (\grosseursimple);

\fill[racine] (0.800000000000000,1.38564064605510) circle (\grosseursimple);
\fill[racine] (1.60000000000000,0.01) circle (\grosseursimple);
\fill[racine] (1.20000000000000,2.07846096908265) circle (\grosseursimple);
\fill[racine] (2.80000000000000,2.07846096908265) circle (\grosseursimple);
\fill[racine] (0.888888888888889,0.769800358919501) circle (\grosseur);
\fill[racine] (2.88888888888889,0.384900179459750) circle (\grosseur);
\fill[racine] (2.22222222222222,2.30940107675850) circle (\grosseur);
\fill[racine] (2.40000000000000,0.01) circle (\grosseursimple);
\fill[racine] (3.11111111111111,0.769800358919501) circle (\grosseur);
\fill[racine] (1.77777777777778,2.30940107675850) circle (\grosseur);
\fill[racine] (1.11111111111111,0.384900179459750) circle (\grosseur);
\fill[racine] (3.20000000000000,1.38564064605510) circle (\grosseursimple);

\draw (1.33333333333333,2.30940107675850) -- (2.66666666666667,2.30940107675850);

\node[label=left :{$s(\alpha_r)$}] (b) at (1.33333333333333,2.30940107675850) {};
\node[label=right :{$s(\alpha_t)$}] (b) at (2.66666666666667,2.30940107675850) {};

\draw (1.33333333333333,2.30940107675850) -- (4,0);

\node[label=right :{$t s(\alpha_r)$}] (b) at (3.11111111111111,0.769800358919501) {};

\end{tikzpicture}	
	
	\caption{Example of a biclosed set that is not biconvex in a geometric representation of dimension $3$ of  the universal Coxeter group $W'=\mathcal U_4$ of rank $4$; the root system is represented as a root subsystem in the classical geometric representation of the universal group $\mathcal U_3$ of rank $3$.  The circle represent the normalized vectors in the isotropic cone.\label{fig:Univ}}
	\end{center}
\end{figure}
\begin{remark} However, in the case of an irreducible affine Coxeter group, Dyer~\cite{dyer_personal_2013}  informed us that he can prove that biclosed sets are precisely biconvex sets.
\end{remark}
\label{ex:41}
\end{ex}

\begin{ex}[A biclosed set that is not biconvex in the classical geometric representation of a Coxeter group of rank $4$] We consider  $(W,S)$ to be the Coxeter system of rank $4$ whose Coxeter graph is: 
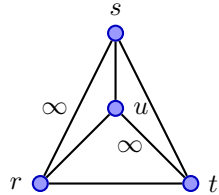
\begin{center}
\begin{tikzpicture}[sommet/.style={inner sep=2pt,circle,draw=blue!75!black,fill=blue!40,thick},]

\node[sommet,label=right:$u$] (u) at (0,0) {};
\node[sommet,label=right:$t$] (t) at (1,-1) {} edge[thick] node[left] {$\infty$} (u);
\node[sommet,label=left:$r$] (r) at (-1,-1) {} edge[thick] (u) edge[thick] (t);
\node[sommet,label=above:$s$] (s) at (0,1) {} edge[thick] (u) edge[thick] (t) edge[thick] node[left] {$\infty$} (r);

\end{tikzpicture}
\end{center}
Using the above example, we show that such a non-biconvex biclosed set lives in the root system of the classical geometric representation of $(W,S)$. Consider $\Delta'=\{r(\alpha_u),r(\alpha_t),s(\alpha_u),s(\alpha_t)\}$ and~$S'=\{rur,rtr,sus,sts\}$. Note that $r(\alpha_u)=\alpha_r+\alpha_u$, $r(\alpha_t)=\alpha_r+\alpha_t$, $s(\alpha_u)=\alpha_s+\alpha_u$, $s(\alpha_t)=\alpha_s+\alpha_t$. We have 
$$
B(r(\alpha_u),r(\alpha_t))=B(s(\alpha_u),s(\alpha_t))=B(r(\alpha_u),s(\alpha_u))=B(s(\alpha_t),r(\alpha_t))=-1
$$
and
$$
B(r(\alpha_u),s(\alpha_t))=B(r(\alpha_t),s(\alpha_u))=-3.
$$
So $(W',S')$ is a universal Coxeter system with geometric representation of dimension $3$ and simple system $\Delta'$ as in~\ref{ex:41}. Let $H=\ker\rho$ be the hyperplane generated by $\Delta'$. We know by a result of Dyer~\cite[Theorem~4.4.]{dyer_reflection_1990} 
that $\Phipprime=\Phip\cap H$ is a positive root system associated to a reflection subgroup of $W$; its associated simple system is the basis of $\cone(\Phipprime)$,  which is equal to $\Delta'$ since $H\cap\cone(\Delta)=\cone(\Delta')$. So $\Phipprime$ is the positive root system of $W'$. Consider now $I=\{rur,rtr,sus\}\subset S'$, so by Example~\ref{ex:41}, $\Phi_I$ is biclosed but not biconvex. Then $A=\PhipI\cup\{\beta\in\Phip\,|\,\rho(\beta)>0\}$ is  biclosed but not biconvex. We illustrate this example in Figure~\ref{fig:Univ2}.

\begin{figure}[!h]
\centering
\begin{tikzpicture}
\node[above right] (img) at (0,0) {\includegraphics[width=0.5\hsize]{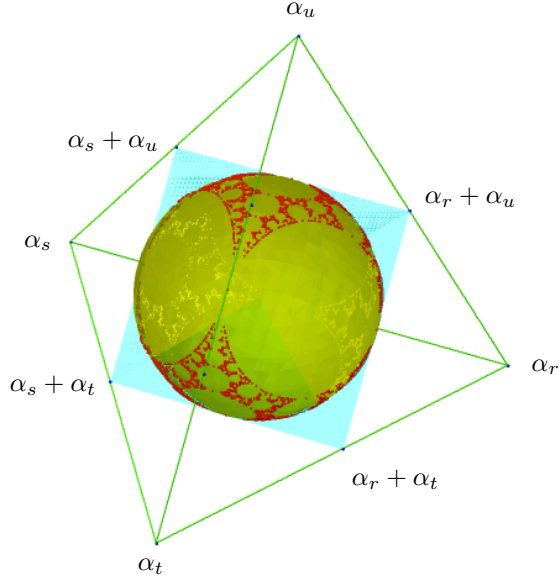}};
\node at (0,4.4) {$\alpha_s$};
\node at (1.5,0.1) {$\alpha_t$};
\node at (0.2,2.5) {$\alpha_s+\alpha_t$};

\node at (3.5,7.5) {$\alpha_u$};
\node at (1,5.75) {$\alpha_s+\alpha_u$};
\node at (5.75,5) {$\alpha_r+\alpha_u$};
\node at (6.75,2.75) {$\alpha_r$};
\node at (4.75,1.25) {$\alpha_r+\alpha_t$};
\end{tikzpicture}
	\caption{Example of a biclosed set that is not biconvex in the classical geometric representation of a Coxeter group of rank $4$.\label{fig:Univ2}}
\end{figure}
\end{ex}

\begin{remark} We do not know of any example of a biclosed set that is not biconvex in a geometric representation of dimension $3$ of a Coxeter system of rank $3$. 
\end{remark}

\subsection{On infinite biconvex and separable sets} We establish in the following example that biconvex sets are not separable.

We consider  $W$ to be of type $\tilde A_3$ generated by $\{s_1,s_2,s_3,s_4\}$, where $(s_1s_3)^2=(s_2s_4)^2=e$. We know that $E=\{\h\delta\}$ is a singleton (the red dot in the center of the tetrahedron in Figure~\ref{fig:AffineA3}). Let $X=\{s_2 s_1 s_3 s_ 2 s_1,s_2 s_1 s_4\}$ and $A=\cone_\Phi(N(X))$. So $\h A$ is the union of all the segments $[\h\beta,\h\delta]\cap\h \Phi$, with $\beta\in N(X)$ (in Figure~\ref{fig:AffineA3}, $\h N(214)$ is the blue triangle on the bottom face and $\h N(21321)$ is in yellow on the left face). One checks that~$A$ is biconvex. But $A$ is not separable since $\delta\in \conv(\h A)\cap \conv(\h A^c)$. One can also check that $A=\overline{N(X)}$. Therefore this example also shows that the {\em infinite} join of $N(X)$, for $X=\{s_2 s_1 s_3 s_2 s_1,s_2 s_1 s_4\}$, exists in $(\mathcal B,\subseteq)$, is not separable and is not the inversion set (nor the complement) of a finite or infinite reduced word. 

\begin{center}
\begin{figure}
\begin{tikzpicture}

\node at (0,0) {\includegraphics[width=8cm]{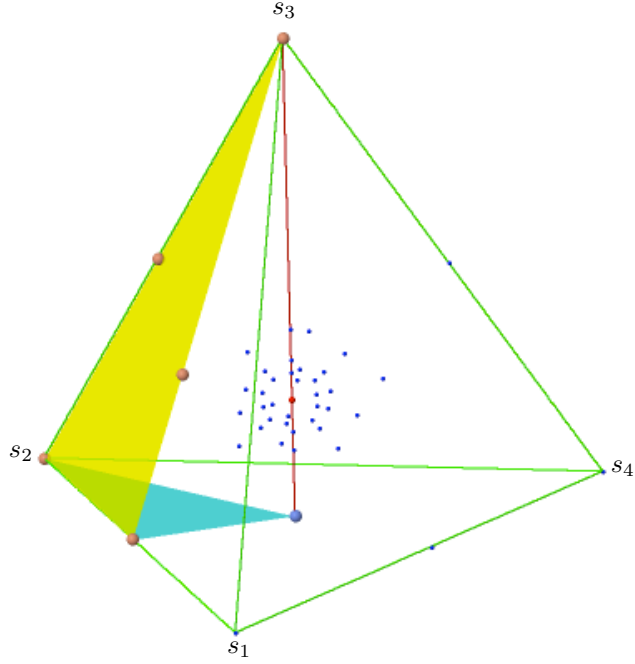}};

\node at (-4,-1.6) {$s_2$};
\node at (-1.1,-4.2) {$s_1$};
\node at (4,-1.8) {$s_4$};
\node at (-0.5,4.3) {$s_3$};
\end{tikzpicture}

\caption{A nonseparable biclosed set in the affine root system of type $\tilde A_3$.\label{fig:AffineA3}}
\end{figure}
\end{center}

\subsection{Infinite biconvex sets in the Kac-Moody/Lie setting} Biclosed sets are often called biconvex, or compatible, in the literature concerned with root systems associated to a Kac-Moody algebra, see for instance~\cite{cellini_structure_1998,ito_classification_2001,ito_parameterizations_2005,baumann_affine_2014,lam_total_2013}. One has to be very careful with these two notions of {\em biconvexity}  since {\em they are not the same in general when applied to infinite sets of positive roots}. Indeed, in the definition of biconvex in the Kac-Moody setting, the complement is not taken in $\Phip$ but in $\Phip\sqcup\Phipim$, where $\Phipim$ denotes the positive imaginary roots. So a subset of $\Phip\sqcup \Phipim$ is said to be biconvex in a Kac-Moody root system if $A$ and $(\Phip\sqcup\Phipim)\setminus A$ are convex (we have to consider imaginary roots as well as real roots). In affine types, it is known that $\Phipim=\mathbb Z_{>0}\delta$,  so $\h{\Phi}^+_{im}=\{\h\delta\}$ is a singleton. Therefore  $\delta$ cannot be in $A$ and its complement. Cellini and Papi~\cite{cellini_structure_1998} proved in the affine case that if $A$ is biconvex then either $A=N(\omega)$ for some $\omega\in \mathcal W^\infty$ if $\delta\notin A$ or $(\Phip\sqcup\Phipim)\setminus A=N(\omega)$ if $\delta\in A$, see also Remark~\ref{rem:CePa}. In particular, this means that the real part of any biconvex set in the sense of the affine Kac-Moody setting is separable in our setting by Proposition~\ref{prop:Infinite}. In  Figure~\ref{fig:AffineA3}, we provide a biconvex set $A$ (in the sense of our definition) that is not separable, so it is not a biconvex set in the sense of a Kac-Moody root system.

\subsection*{Acknowledgment} A significant part of this work was done while the first author (CH) was on sabbatical leave at LIX at \'Ecole Polytechnique, Paris, in September-October 2013 and at Institut de Recherche Math\'ematique Avanc\'ee (IRMA), Universit\'e de Strasbourg, from October 2013 to June 2014. The first author thanks N.~Ressayre to have brought to his attention the reference~\cite{kostant_lie_1961}.  The first author wishes also to warmly thank Matthew Dyer for many fruitful and motivating private communications~\cite{dyer_personal_2013}, to have inspired the example given in~\S\ref{sse:BiclosBiconv} and for the example given in~Figure~\ref{fig:AffineA3}. The second author would like to thank Karim Adiprasito, Hao Chen, Emerson Leon, Vivien Ripoll and G\"unter M. Ziegler for many important discussions on convex geometry throughout the writing of this article. The authors are grateful to two anonymous referees, whose comments helped to significantly improve the content of this article.




\end{document}